\documentclass{amsart}

\usepackage{amsthm,amsmath,amssymb,amsfonts}
\usepackage{pb-diagram}
\usepackage[colorlinks=true,hypertexnames=false,linkcolor=blue,citecolor=blue]{hyperref}
\usepackage{epsfig,graphics}
\usepackage{amscd}
\usepackage{amsmath}
\usepackage{bm}
\usepackage{amssymb}
\usepackage{graphicx}
\usepackage{verbatim}
\usepackage[percent]{overpic}

\newtheorem{theorem}{Theorem}[section]

\newtheorem{lema}[theorem]{Lemma}
\newtheorem{lem}[theorem]{Lemma}
\newtheorem{claim}[theorem]{Claim}
\newtheorem{prop}[theorem]{Proposition}
\newtheorem{coro}[theorem]{Corollary}

\theoremstyle{definition}
\newtheorem{definition}[theorem]{Definition}
\newtheorem{defn}[theorem]{Definition}

\theoremstyle{remark}
\newtheorem{remark}[theorem]{Remark}

\numberwithin{equation}{section}

\theoremstyle{plain}
\newtheorem{maintheorem}{Theorem}

\newcommand{\D}{\ensuremath{\mathbb{D}}}
\newcommand{\Q}{\ensuremath{\mathbb{Q}}}
\newcommand{\R}{\ensuremath{\mathbb{R}}}
\newcommand{\Z}{\ensuremath{\mathbb{Z}}}
\newcommand{\C}{\ensuremath{\mathbb{C}}}
\newcommand{\nt}{\ensuremath{\mathbb{N}}}

\DeclareMathOperator{\Id}{Id}
\DeclareMathOperator{\Diff}{Diff}
\DeclareMathOperator{\Hom}{Hom}
\DeclareMathOperator{\Aff}{Aff}

\DeclareMathOperator{\dist}{dist}
\DeclareMathOperator{\Dom}{Dom}
\DeclareMathOperator{\id}{id}

\begin{document}

\title{Rigidity of critical circle maps}

\author{Pablo Guarino}
\address{Instituto de Matem\'atica e Estat\'istica, Universidade Federal Fluminense}
\curraddr{Rua M\'ario Santos Braga S/N, 24020-140, Niter\'oi, Rio de Janeiro, Brazil}
\email{pablo\_\,guarino@id.uff.br}

\author{Marco Martens}
\address{Department of Mathematics, Stony Brook University, Stony Brook NY, USA}
\email{marco@math.sunysb.edu}

\author{Welington de Melo}
\address{IMPA, Rio de Janeiro, Brazil}
\curraddr{Estrada Dona Castorina 110, 22460-320}
\email{demelo@impa.br}

\thanks{The authors are grateful to Edson de Faria for several conversations about the subject. During the preparation of this article, M. M. visited IMPA and P. G. visited Stony Brook. We are grateful to both institutions for their hospitality. P. G. was partially supported by FAPESP Grant 2012/06614-8 and by FAPERJ Grant E-26/102.784/2012. W. M. was partially supported by CNPq grant 300813/2010-4.}



\begin{abstract} We prove that any two $C^4$ critical circle maps with the same irrational rotation number and the same odd criticality are conjugate to each other by a $C^1$ circle diffeomorphism. The conjugacy is $C^{1+\alpha}$ for Lebesgue almost every rotation number.
\end{abstract}

\maketitle

\section{Introduction}

By a \emph{critical circle map} we mean an orientation preserving $C^4$ circle homeomorphism, with exactly one non-flat critical point of odd criticality (see Definition \ref{defccm} below). In 1984 Yoccoz proved that if such a critical circle map has no periodic points, all its orbits are dense \cite{yoccoz}. This implies the following topological rigidity result: if two critical circle maps have the same irrational rotation number, then there exists a unique conjugacy between them that sends the critical point to the critical point.

Numerical observations (\cite{feigetal}, \cite{ostlundetal}, \cite{sh}) suggested in the early eighties that this topological conjugacy could be, in fact, a smooth diffeomorphism, at least for bounded combinatorics. These observations led to the \emph{rigidity conjecture}, posed in several works by Lanford (\cite{lanford1}, \cite{lanford2}), Rand (\cite{rand1}, \cite{rand2} and \cite{rand3}, see also \cite{ostlundetal}) and Shenker (\cite{sh}, see also \cite{feigetal}) among others. Our main result is the following:

\begin{maintheorem}[Rigidity]\label{rigidity} Let $f$ and $g$ be two $C^4$ circle homeomorphisms with the same irrational rotation number and with a unique critical point of the same odd type. Let $h$ be the unique topological conjugacy between $f$ and $g$ that maps the critical point of $f$ to the critical point of $g$. Then:
\begin{enumerate}
\item\label{Aitem1} $h$ is a $C^1$ diffeomorphism.
\item\label{Aitem2} $h$ is $C^{1+\alpha}$ at the critical point of $f$ for a universal $\alpha>0$.
\item\label{Aitem3} For a full Lebesgue measure set of rotation numbers, $h$ is a $C^{1+\alpha}$ diffeomorphism.
\end{enumerate}
\end{maintheorem}

See \cite[Section 4.4]{dfdm1} for the definition of the full measure set of rotation numbers considered in Conclusion \eqref{Aitem3} of Theorem \ref{rigidity}. Let us point out that, by a result of \'Avila \cite{avila}, there exist two real analytic critical circle maps with the same irrational rotation number and the same criticality that are not $C^{1+\beta}$ conjugate for any $\beta>0$.

Many papers have addressed the rigidity problem: see \cite{edson}, \cite{dfdm1}, \cite{edsonwelington2}, \cite{yampolsky1}, \cite{yampolsky2}, \cite{yampolsky3}, \cite{yampolsky4}, \cite{avila}, \cite{KY}, \cite{khaninteplinsky} and \cite{guamelo}. In particular, Theorem \ref{rigidity} was proven for real analytic critical circle maps by a series of papers by de Faria-de Melo, Khmelev-Yampolsky and Khanin-Teplinsky (\cite{edsonwelington2}, \cite {KY} and \cite{khaninteplinsky}).

Moreover, in the $C^3$ category rigidity holds for bounded combinatorics: any two $C^3$ critical circle maps with the same irrational rotation number of \emph{bounded type} and the same odd criticality are conjugate to each other by a $C^{1+\alpha}$ circle diffeomorphism, for some universal $\alpha>0$ \cite{guamelo}. Let us mention that this was the precise statement of the rigidity conjecture mentioned above.

\begin{remark} We do not know whether Theorem \ref{rigidity} holds for $C^3$ dynamics with unbounded combinatorics, and we also do not know if rigidity holds on less regularity, for instance $C^{2+\alpha}$ smoothness, even for bounded combinatorics. Moreover, we do not know how to deal with critical points of non-integer criticality, not even with fractional criticality (see Definition \ref{defccm} below).
\end{remark}

By the famous rigidity result of Herman, \cite {hem}, improved by Yoccoz \cite{yoccoz0}, any real analytic circle diffeomorphisms whose rotation number satisfies a Diophantine condition is real analytic conjugate to a rotation.  We believe that there exist two real analytic critical circle maps with the same rotation number of bounded type and the same criticality that are not $C^2$ conjugate.

Theorem \ref{rigidity} follows from the following theorem, which is our main result on the dynamics of the renormalization operator acting on the $C^4$ class:

\begin{maintheorem}[Exponencial convergence in the $C^2$-distance]\label{expconv} There exists a universal constant $\lambda \in (0,1)$ such that given two $C^4$ critical circle maps $f$ and $g$ with the same irrational rotation number and the same criticality, there exists $C=C(f,g)>0$ such that for all $n\in\nt$ we have:$$d_2\big(\mathcal{R}^n(f),\mathcal{R}^n(g)\big) \leq C\lambda^n\,,$$where $d_2$ is the $C^2$ distance in the space of $C^2$ critical commuting pairs.
\end{maintheorem}

This paper is devoted to prove Theorem \ref{expconv}. The fact that Theorem \ref{expconv} implies Theorem \ref{rigidity} follows from well-known results by de Faria-de Melo \cite[First Main Theorem, page 341]{dfdm1} and Khanin-Teplinsky \cite[Theorem 2, page 198]{khaninteplinsky}, and it will be explained in Section \ref{finalmesmo}.

\subsection{Strategy of the proof of Theorem \ref{expconv}}\label{S:sketch} A $C^4$ critical circle map $f$ with irrational rotation number generates a sequence $\big\{\mathcal{R}^n(f)\big\}_{n\in\nt}$ of commuting pairs of interval maps, each one being the renormalization of the previous one (see Definition \ref{renop}). To prove Theorem \ref{expconv} we need to prove the exponential convergence in the $C^2$ distance of the orbits generated by two $C^4$ critical circle maps with the same irrational combinatorics and the same criticality. Roughly speaking, the proof has three ingredients:
\begin{enumerate}
\item\label{ing1} The existence of a sequence $\{f_n\}_{n\in\nt}$ that belongs to a universal $C^{\omega}$-compact set of real analytic critical commuting pairs, such that $\mathcal{R}^n(f)$ is $C^3$ exponentially close to $f_n$ at a universal rate, and both have the same rotation number (Theorem \ref{compacto}).
\item\label{ing2} The uniform exponential contraction of renormalization when restricted to topological conjugacy classes of real analytic critical commuting pairs (Theorem \ref{uniform}).
\item\label{ing3} The \emph{key lemma} (Lemma \ref{main}): a Lipschitz estimate for the renormalization operator, when restricted to suitable bounded pieces of topological conjugacy classes of infinitely renormalizable $C^3$ critical commuting pairs with negative Sch\-warzian derivative.
\end{enumerate}

The fact that \eqref{ing1}, \eqref{ing2} and \eqref{ing3} imply Theorem \ref{expconv} will be proved in Section \ref{final}. Theorem \ref{compacto} will be obtained in Section \ref{melhoratese}, based on a previous construction developed by two of the authors in \cite{guamelo}. Theorem \ref{uniform} was proved by de Faria and de Melo \cite{edsonwelington2} for rotation numbers of bounded type, and extended by Khmelev and Yampolsky \cite{KY} to cover all irrational rotation numbers.

Our main task in this paper is to prove Lemma \ref{main}, and its proof will occupy Sections \ref{Smonfam} to \ref{lip}. For bounded combinatorics, it is not difficult to prove the key lemma (see \cite[Section 3.5]{guamelo}). Let us explain here the main difficulties in the unbounded case: given an infinitely renormalizable critical commuting pair (see Definition \ref{critpair}) one obtains the domain of its first renormalization by iterating a boundary point of its original domain (see Definition \ref{renop}). The rotation number of the map determines the number of iterates involved to obtain the boundary of the renormalization. The number, denoted by $a$, plays a delicate role. Denote the obtained boundary point by $x_a$.

Under the $K$-controlled and the negative Schwarzian assumptions (see Section \ref{Seccont} for definitions), the geometry of this piece of monotone orbit is precisely described by a result due to J.-C. Yoccoz (see Lemma \ref{doyoc} in Section \ref{Ssap}): for large values of $a$, most of the points are concentrated in a small area of accumulation, with size comparable to $1/\sqrt{a}$.

Now consider two $C^3$ critical commuting pairs which are renormalizable with the same period $a\in\nt$. Suppose their $C^2$-distance is $\varepsilon>0$. Denote the corresponding new boundary points by $x_a^0$ and $x_a^1$. This new boundary points can indeed be arbitrarily separated. A consequence of this is that renormalization is definitely {\it not} even H\"older continuous on the space of \emph{all} commuting pairs. However, we will prove that when one renormalizes two pairs with the same irrational rotation number, the new boundary points become close. Namely, $|x_a^0-x_a^1|=O(\varepsilon)$. The systems are said to be \emph{synchronized}, see Definition \ref{defsync} and the \emph{Synchronization Lemma} in Section \ref{sync}.

If a separation would occur, $|x_a^0-x_a^1|$ not small, then the $a$-values of future renormalization will not be the same. This contradicts the fact that the systems had the same irrational rotation number. 

Observe when the pairs are very close, the $a$-values of arbitrarily many future renormalizations will play a role. From this it seems to be an impossible task, to obtain an estimate for the distance between the new boundary points. A natural notion of order is introduced in Section \ref{order}. One can interpret this order as a cone-field associated with the unstable direction of renormalization. The rotation number is monotone with respect to the order. Using the order, synchronization follows.

Synchronization implies that the pieces of orbit involved in defining the boundaries of the domains of the renormalization of both pairs are everywhere close together (Proposition \ref{lipC2C2amigos}). This allows to control the branches of the renormalizations.  Hence, renormalization is uniformly Lipschitz on classes of \emph{controlled} pairs with the same irrational rotation number (see Section \ref{lip}).

\begin{remark} The renormalization theory for critical circle maps was developed during the late seventies and the eighties (see \cite{feigetal}, \cite{lanford1}, \cite{lanford2}, \cite{ostlundetal}, \cite{rand1}, \cite{rand2}, \cite{rand3} and \cite{sh}), in parallel with the renormalization theory of \emph{unimodal maps} of the interval (see \cite[Chapter VI]{dmvs} and the references therein, see also \cite{mcicm} and \cite{dmicm}). The fact that the exponential convergence of renormalization for real analytic dynamics implies the exponential convergence for finitely smooth unimodal maps, was obtained by the third author and Pinto in the late nineties \cite{demelopinto}. Their methods, however, do not apply for critical circle maps, not even for bounded combinatorics.

One source for this difference is the fact that, for infinitely renormalizable unimodal maps, the sum of the lengths of the (cycle of) intervals related to each level of renormalization goes to zero exponentially fast. This gives a strong control of the non-linearity, even for $C^2$ dynamics. In the circle case, however, the intervals involved at each step of renormalization cover the whole circle (see Section \ref{Sprel}). In this case $C^4$ smoothness is needed in order to have $C^3$-bounded orbits of renormalization (see Theorem \ref{realB}) to be able to control the non-linearity with the help of the Schwarzian derivative.

Another crucial ingredient in \cite{demelopinto} is the \emph{hybrid} lamination, and the fact that its holonomy is quasi-conformal (see \cite{lyubich} and the references therein). To the best of our knowledge, no similar structure has been obtained for critical circle maps.
\end{remark}

This paper is organized as follows: in Section \ref{Sprel} we briefly review some preliminaries, and introduce the renormalization operator acting on critical commuting pairs. In Section \ref{Seccont} we define \emph{controlled} commuting pairs (see Definition \ref{defcont}) and state some of its properties. In Section \ref{Ssap} we state Lemma \ref{main} discussed above (the key lemma), and we also state the Yoccoz's lemma already mentioned (Lemma \ref{doyoc}).

In Section \ref{Smonfam} we construct a suitable one-parameter family around a critical commuting pair, the \emph{standard family}. These families are \emph{transversal} to topological classes, and they may be regarded as \emph{unstable manifolds} for renormalization (see especially Proposition \ref{geomprop}, where we estimate the expansion of renormalization along these families).

In Section \ref{orbdef} we introduce the notion of \emph{synchronization} (Definition \ref{defsync}) and we collect several estimates that hold under synchronization (see especially Propositions \ref{Dxi} and \ref{Di}). In Section \ref{comp} we prove that the key lemma (Lemma \ref{main}) holds under the synchronization assumption (see Proposition \ref{lipC2C2amigos}). In Section \ref{order} we define the already discussed notion of \emph{order}. In Section \ref{sync} we prove the \emph{Synchronization Lemma}, and in Section \ref{lip} we finally prove Lemma \ref{main}.

As we said before, in Section \ref{melhoratese} we prove Theorem \ref{compacto}, and in Section \ref{final} we finally prove that Theorem \ref{expconv} follows from Theorem \ref{compacto}, Theorem \ref{uniform} and Lemma \ref{main}. In Section \ref{finalmesmo} we give precise references of the fact that Theorem \ref{expconv} implies Theorem \ref{rigidity}, and in Appendix \ref{apA} we prove Proposition \ref{holder}, stated and used in Section \ref{melhoratese}.

Let us fix some notation that we will use along this paper: $\nt$, $\Z$, $\Q$, $\R$ and $\C$ denotes respectively the set of natural, integer, rational, real and complex numbers. With $S^1$ we denote the multiplicative group of complex numbers of modulus one, that is, the unit circle $S^1=\big\{z\in\C:|z|=1\big\}$. We will identify $S^1$ with $\R/\Z$ under the universal covering map $\pi:\R \to S^1$ given by $\pi(t)=\exp(2\pi it)$. We denote by $\Hom_{+}(S^1)$ the group (under composition) of orientation preserving circle homeomorphisms, and by $\Diff_{+}^{r}(S^1)$ its subgroup of $C^r$ diffeomorphisms for any $r \geq 1$. The function $\rho:\Hom_{+}(S^1)\to[0,1)$ will denote the rotation number function. The Euclidean length of an interval $I$ will be denoted by $|I|$. Uniform constants will always be denoted by $K$. Their value might change during estimates.

\section{Preliminaries}\label{Sprel}

\begin{definition}\label{defccm} By a \emph{critical circle map} we mean an orientation preserving $C^4$ circle homeomorphism $f$ with exactly one critical point $c$ such that, in a neighbourhood of $c$, the map $f$ can be written as $f(t)=f(c)+\big(\phi(t)\big)^{2d+1}$, where $\phi$ is a $C^4$ local diffeomorphism with $\phi(c)=0$, and $d\in\nt$ with $d \geq 1$. The \emph{criticality} (or \emph{order}, or \emph{type}, or \emph{exponent}) of the critical point $c$ is the odd integer $2d+1$. We also say that the critical point $c$ is \emph{non-flat}.
\end{definition}

As an example, consider the so-called \emph{Arnold's family}, which is the one-parameter family $\widetilde{f}_{\omega}:\R\to\R$ given by:
$$\widetilde{f}_{\omega}(t)=t+\omega-\left(\frac{1}{2\pi}\right)\sin(2\pi t)\quad\mbox{for $\omega\in[0,1)$.}$$

Since each $\widetilde{f}_{\omega}$ commutes with unitary translation, it is the lift, under the universal cover $t \mapsto e^{2\pi it}$, of an orientation preserving real analytic circle homeomorphism, presenting one critical point of cubic type at $1$, the projection of the integers (see also \cite[Section 6]{defgua} for examples of rational maps).

We will assume along this article that the rotation number $\rho(f)=\theta$ in $[0,1)$ is irrational, and let$$\big[a_0,a_1,...,a_n,a_{n+1},...\big]$$be its continued fraction expansion:$$\theta=\lim_{n\to+\infty}\dfrac{1}{a_0+\dfrac{1}{a_1+\dfrac{1}{a_2+\dfrac{1}{\ddots\dfrac{1}{a_n}}}}}\,.$$

We define recursively the \emph{return times} of $\theta$ by:
\begin{center}
$q_0=1,\quad$ $\quad q_1=a_0\quad$ and $\quad q_{n+1}=a_nq_n+q_{n-1}\quad$ for $\quad n \geq 1$.
\end{center}

Recall that the numbers $q_n$ are also obtained as the denominators of the truncated expansion of order $n$ of $\theta$:$$\frac{p_n}{q_n}=[a_0,a_1,a_2,...,a_{n-1}]=\dfrac{1}{a_0+\dfrac{1}{a_1+\dfrac{1}{a_2+\dfrac{1}{\ddots\dfrac{1}{a_{n-1}}}}}}\,\,.$$

\subsection{Real bounds}\label{S:realbounds} Denote by $I_n$ the interval $[c,f^{q_n}(c)]$ and define $\mathcal{P}_n$ as:$$\mathcal{P}_n=\big\{I_n,f(I_n),...,f^{q_{n+1}-1}(I_n)\big\} \bigcup \big\{I_{n+1},f(I_{n+1}),...,f^{q_n-1}(I_{n+1})\big\}$$

A crucial combinatorial fact is that $\mathcal{P}_n$ is a partition (modulo boundary points) of the circle for every $n \in \nt$. We call it the \emph{n-th dynamical partition} of $f$ associated with the point $c$. Note that the partition $\mathcal{P}_n$ is determined by the piece of orbit:$$\{f^{j}(c): 0 \leq j \leq q_n + q_{n+1}-1\}$$

The transitions from $\mathcal{P}_n$ to $\mathcal{P}_{n+1}$ can be described in the following easy way: the interval $I_n=[c,f^{q_n}(c)]$ is subdivided by the points $f^{jq_{n+1}+q_n}(c)$ with $1 \leq j \leq a_{n+1}$ into $a_{n+1}+1$ subintervals. This sub-partition is spreaded by the iterates of $f$ to all the $f^j(I_n)=f^j([c,f^{q_n}(c)])$ with $0 \leq j < q_{n+1}$. The other elements of the partition $\mathcal{P}_n$, which are the $f^j(I_{n+1})$ with $0 \leq j < q_n$, remain unchanged.

As we are working with critical circle maps with a single critical point, our partitions in this article are always determined by the critical orbit.

\begin{theorem}[real bounds]\label{realbounds} There exists $K>1$ such that given a $C^3$ critical circle map $f$ with irrational rotation number there exists $n_0=n_0(f)$ such that for all $n \geq n_0$ and for every pair $I,J$ of adjacent atoms of $\mathcal{P}_n$ we have:$$K^{-1}|I| \leq |J| \leq K|I|.$$
Moreover, if $Df$ denotes the first derivative of $f$, we have:$$\frac{1}{K}\leq\frac{\big|Df^{q_n-1}(x)\big|}{\big|Df^{q_n-1}(y)\big|}\leq K\quad\mbox{for all $x,y \in f(I_{n+1})$ and for all $n \geq n_0$, and:}$$
$$\frac{1}{K}\leq\frac{\big|Df^{q_{n+1}-1}(x)\big|}{\big|Df^{q_{n+1}-1}(y)\big|}\leq K\quad\mbox{for all $x,y \in f(I_{n})$ and for all $n \geq n_0$.}$$
\end{theorem}

Theorem \ref{realbounds} was proved by \'Swi\c{a}tek and Herman (see \cite{herman}, \cite{swiatek}, \cite{graswia} and \cite{dfdm1}). The control on the distortion of the return maps follows from Koebe distortion principle (see \cite[Section 3]{dfdm1}). Note that for a rigid rotation we have $|I_n|=a_{n+1}|I_{n+1}|+|I_{n+2}|$. If $a_{n+1}$ is big, then $I_n$ is much larger than $I_{n+1}$. Thus, even for rigid rotations, real bounds do not hold in general.

\subsection{Critical commuting pairs}\label{subccp}

\begin{definition}\label{critpair} A $C^r$ \emph{critical commuting pair} $\zeta=(\eta,\xi)$ consists of two $C^r$ orientation preserving homeomorphisms $\eta:I_{\eta}\to\eta(I_{\eta})$ and $\xi:I_{\xi}\to\xi(I_{\xi})$ where:
\begin{enumerate}
\item $I_{\eta}=[0,\xi(0)]$ and $I_{\xi}=[\eta(0),0]$ are compact intervals in the real line;
\item\label{comor} $\big(\eta\circ\xi\big)(0)=\big(\xi\circ\eta\big)(0) \neq 0$;
\item\label{la4} $D\eta(x)>0$ for all $x \in I_{\eta}\!\setminus\!\{0\}$ and $D\xi(x)>0$ for all $x \in I_{\xi}\!\setminus\!\{0\}$;
\item\label{pc} The origin is a non-flat critical point for both $\eta$ and $\xi$ with the same odd criticality, that is, there exist a positive integer $d$, an open interval $C$ around the origin and two orientation preserving $C^r$ diffeomorphisms $\phi:C\to\phi(C)$ and $\psi:C\to\psi(C)$ fixing the origin such that $\eta(x)=\big(\phi(x)\big)^{2d+1}\!+\eta(0)$ for all $x \in C \cap I_{\eta}$ and $\xi(x)=\big(\psi(x)\big)^{2d+1}\!+\xi(0)$ for all $x \in C \cap I_{\xi}$;
\item\label{la5} The left-derivatives of the composition $\eta\circ\xi$ at the origin coincide with the corresponding right-derivatives of the composition $\xi\circ\eta$ at the origin, that is, for each $j\in\{1,2,...,r\}$ we have $D_{-}^j\big(\eta\circ\xi\big)(0)=D_{+}^j\big(\xi\circ\eta\big)(0)$.
\end{enumerate}
\end{definition}

\begin{figure}[h]\label{fig1compair}
\begin{center}
\includegraphics[scale=0.8]{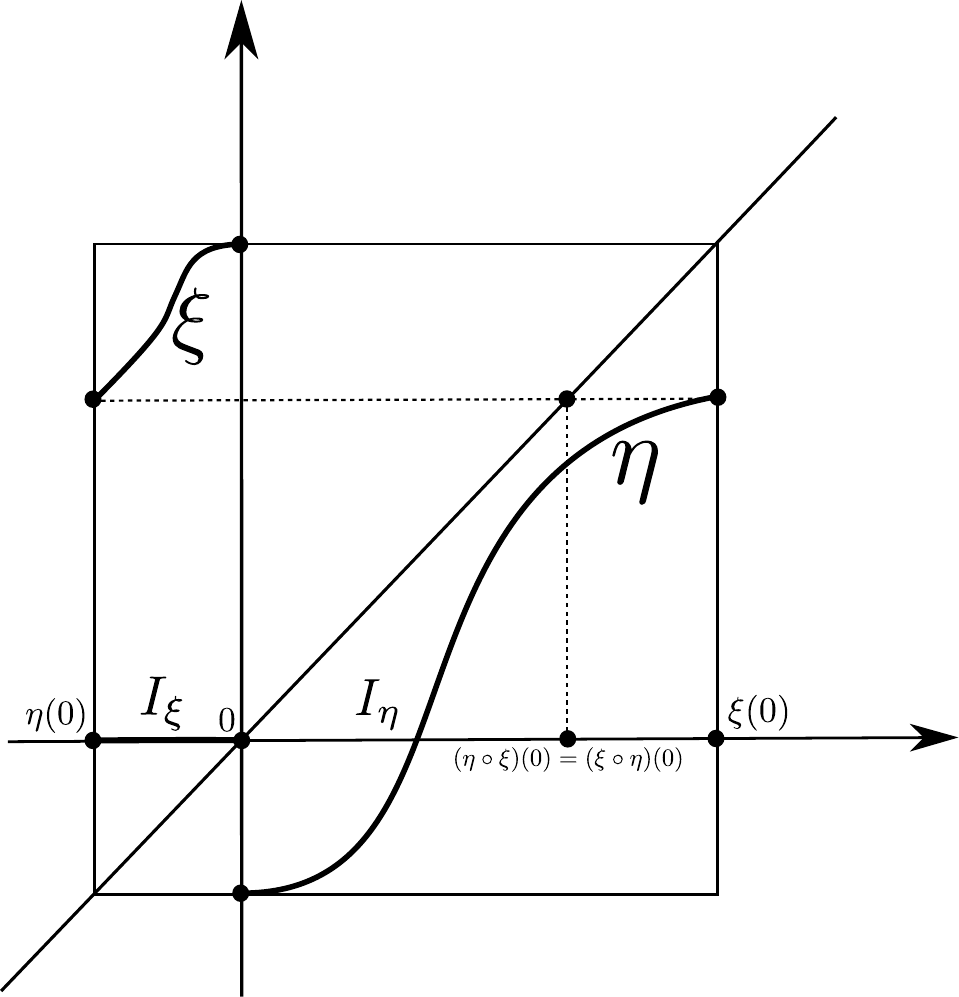}
\caption{A commuting pair.}
\end{center}
\end{figure}

Any critical circle map $f$ with irrational rotation number $\theta$ induces a sequence of critical commuting pairs in a natural way: let $\widehat{f}$ be the lift of $f$ to the real line (for the canonical covering $t \mapsto e^{2\pi it}$) satisfying $D\widehat{f}(0)=0$ and $0<\widehat{f}(0)<1$. For each $n \geq 1$ let $\widehat{I}_n$ be the closed interval in the real line, adjacent to the origin, that projects under $t \mapsto e^{2\pi it}$ to $I_n$. Let $T:\R \to \R$ be the translation $x \mapsto x+1$, and define $\eta:\widehat{I}_n \to \R$ and $\xi:\widehat{I}_{n+1} \to \R$ as:$$\eta=T^{-p_{n+1}}\circ\widehat{f}^{q_{n+1}}\quad\mbox{and}\quad\xi=T^{-p_n}\circ\widehat{f}^{q_n}\,\,.$$

It is not difficult to check that $(\eta|_{\widehat{I}_n}, \xi|_{\widehat{I}_{n+1}})$ is a critical commuting pair, that we denote by $(f^{q_{n+1}}|_{I_n},f^{q_n}|_{I_{n+1}})$ to simplify notation.

\begin{figure}[h]\label{fig2compair}
\begin{center}
\includegraphics[scale=0.5]{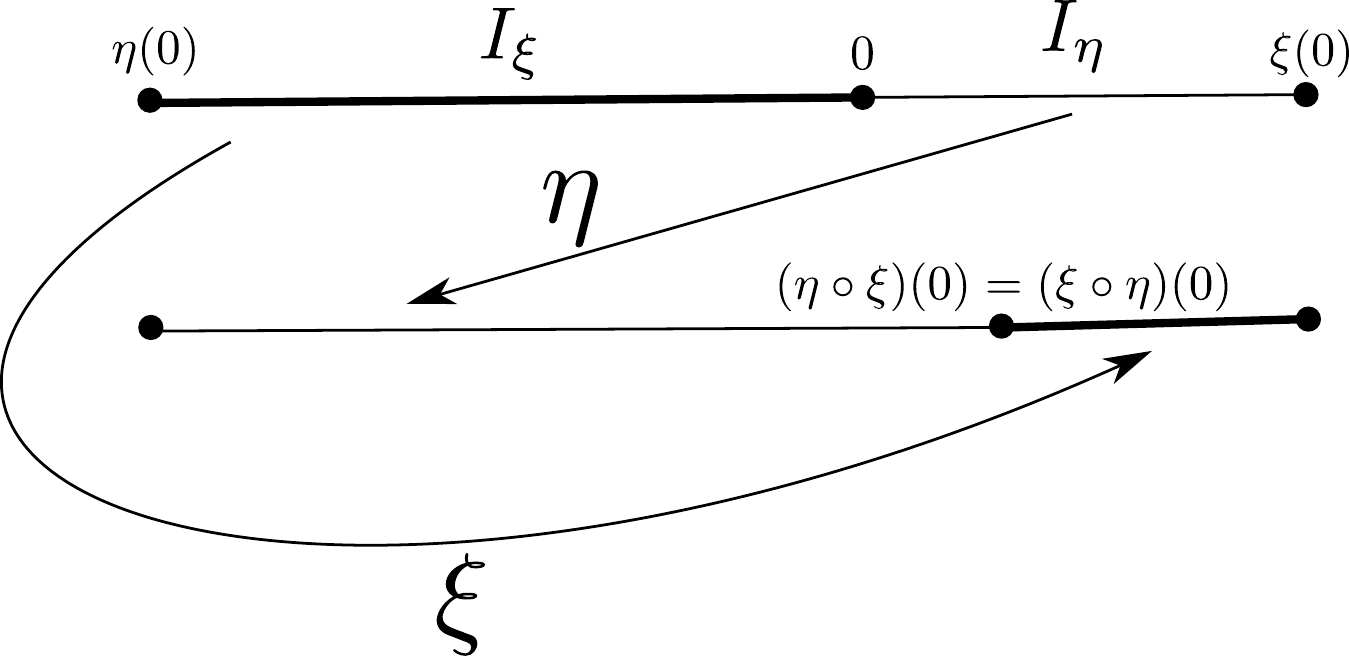}
\caption{Scheme of a commuting pair.}
\end{center}
\end{figure}

As pointed out in \cite[page 344]{dfdm1}, the commuting condition \eqref{comor} in Definition \ref{critpair} actually holds on an open interval around the origin:

\begin{lema}\label{laP4} There exist open intervals $V_{-}\supseteq I_{\xi}$ and $V_{+}\supseteq I_{\eta}$ and $C^r$ homeomorphic extensions $\widehat{\xi}:V_{-}\to\widehat{\xi}(V_{-})\subset\R$ and $\widehat{\eta}:V_{+}\to\widehat{\eta}(V_{+})\subset\R$ of $\xi$ and $\eta$ respectively, satisfying $\big(\widehat{\eta}\circ\widehat{\xi}\big)(x)=\big(\widehat{\xi}\circ\widehat{\eta}\big)(x)$ for all $x$ in the open interval $C$ around the origin given by $C=\big\{x \in V_{-} \cap V_{+}:\widehat{\eta}(x) \in V_{-}\mbox{ and }\widehat{\xi}(x) \in V_{+}\big\}$.
\end{lema}

\begin{proof}[Proof of Lemma \ref{laP4}] Since the origin is a non-flat critical point of odd criticality there exists an open interval $C$ around it on which we can extend both $\eta$ and $\xi$ to $C^r$ homeomorphisms $\widehat{\eta}:C \to A$ and $\widehat{\xi}:C \to B$, where $A$ is an open interval around $\eta(0)$ and $B$ is an open interval around $\xi(0)$ (we may suppose that $A$, $B$ and $C$ are pairwise disjoint). Moreover, since the criticality of both $\widehat{\eta}$ and $\widehat{\xi}$ at the origin is the same odd integer, the composition $\widehat{\xi}\circ\widehat{\eta}^{-1}:A \to B$ is actually a $C^r$ diffeomorphism.

Let $V_{-}=A \cup I_{\xi} \cup C$, which is an open interval where $I_{\xi}$ is compactly contained, and in the same way let $V_{+}=C \cup I_{\eta} \cup B$.

Since the composition $\eta\circ\xi$ is already defined at the left part of $C$, the extension of $\eta$ defined above (given by the non-flatness of the critical point) allows us to extend $\xi$ to the left part of $A$ in the following way: for any $y \in A$ there exists a unique $x \in C$ such that $\widehat{\eta}(x)=y$ (since $A=\widehat{\eta}(C)$ and $\widehat{\eta}:C \to A$ is invertible) and then we define $\widehat{\xi}:A\to\R$ as $\widehat{\xi}(y)=\eta\big(\xi(x)\big)=\big(\eta\circ\xi\circ\widehat{\eta}^{-1}\big)(y)$ if $y<\eta(0)$ and $\widehat{\xi}(y)=\xi(y)$ if $y\geq\eta(0)$.

By Condition \eqref{la5} in Definition \ref{critpair}, the left-derivatives of the composition $\eta\circ\xi\circ\widehat{\eta}^{-1}$ at the point $\eta(0)$ coincide with the corresponding right-derivatives of $\xi$ at $\eta(0)$, that is, $\widehat{\xi}$ is of class $C^r$ at the point $\eta(0)$ (and therefore on the whole domain $V_{-}$). Note also that $\widehat{\xi}$ has no critical points on $V_{-}\!\setminus\!\{0\}$ since $\widehat{\xi}\circ\widehat{\eta}^{-1}:A \to B$ is a $C^r$ diffeomorphism and $\eta$ has no critical points in $B \cap I_{\eta}$ by Condition \eqref{la4}.

In the same way, since the composition $\xi\circ\eta$ is already defined at the right part of $C$ and since $\xi$ is also defined on $C$, we extend $\eta$ to the right part of $B$ by imposing the commuting condition $\widehat{\eta}\circ\widehat{\xi}=\widehat{\xi}\circ\widehat{\eta}$ on $C$ as before.
\end{proof}

\subsection*{The M\"obius metric} Given two critical commuting pairs $\zeta_1=(\eta_1,\xi_1)$ and $\zeta_2=(\eta_2,\xi_2)$ let $A_1$ and $A_2$ be the M\"obius transformations such that for $i=1,2$:$$A_i\big(\eta_i(0)\big)=-1,\quad A_i(0)=0\quad\mbox{and}\quad A_i\big(\xi_i(0)\big)=1\,.$$

\begin{definition}\label{Crmetric} For any $0 \leq r < \infty$ define the $C^r$ metric on the space of $C^r$ critical commuting pairs in the following way:$$d_r(\zeta_1,\zeta_2)=\max\left\{\left|\frac{\xi_1(0)}{\eta_1(0)}-\frac{\xi_2(0)}{\eta_2(0)}\right|, \big\|A_1 \circ \zeta_1 \circ A_1^{-1}-A_2 \circ \zeta_2 \circ A_2^{-1}\big\|_r\right\}$$where $\|\cdot\|_r$ is the $C^r$-norm for maps in $[-1,1]$ with one discontinuity at the origin, and $\zeta_i$ is the piecewise map defined by $\eta_i$ and $\xi_i$:$$\zeta_i:I_{\xi_i} \cup I_{\eta_i} \to I_{\xi_i} \cup I_{\eta_i} \quad\mbox{such that}\quad \zeta_i|_{I_{\xi_i}}=\xi_i \quad\mbox{and}\quad \zeta_i|_{I_{\eta_i}}=\eta_i$$
\end{definition}

When we are dealing with real analytic critical commuting pairs, we consider the $C^{\omega}$-topology defined in the usual way: we say that $\big(\eta_n,\xi_n\big)\to\big(\eta,\xi\big)$ if there exist two open sets $U_{\eta} \supset I_{\eta}$ and $U_{\xi} \supset I_{\xi}$ in the complex plane and $n_0\in\nt$ such that $\eta$ and $\eta_n$ for $n \geq n_0$ extend continuously to $\overline{U_{\eta}}$, are holomorphic in $U_{\eta}$ and we have $\big\|\eta_n-\eta\big\|_{C^0(\overline{U_{\eta}})} \to 0$, and such that $\xi$ and $\xi_n$ for $n \geq n_0$ extend continuously to $\overline{U_{\xi}}$, are holomorphic in $U_{\xi}$ and we have $\big\|\xi_n-\xi\big\|_{C^0(\overline{U_{\xi}})} \to 0$. We say that a set $\mathcal{C}$ of real analytic critical commuting pairs is closed if every time we have $\{\zeta_n\}\subset\mathcal{C}$ and $\{\zeta_n\}\to\zeta$, we have $\zeta\in\mathcal{C}$. This defines a Hausdorff topology, stronger than the $C^r$-topology for any $0 \leq r \leq \infty$ (in particular any $C^{\omega}$-compact set of real analytic critical commuting pairs is certainly $C^r$-compact also, for any $0 \leq r \leq \infty$).

\subsection*{The affine metric} Given two critical commuting pairs $\zeta_1=(\eta_1,\xi_1)$ and $\zeta_2=(\eta_2,\xi_2)$ let $L_{\eta_1}:[0,1] \to I_{\eta_1}$ be $L_{\eta_1}(t)=|I_{\eta_1}|t$, and let $L_{\xi_1}:[-1,0] \to I_{\xi_1}$ be $L_{\xi_1}(t)=|I_{\xi_1}|t$. Define in the same way $L_{\eta_2}$ and $L_{\xi_2}$ and consider:
\begin{align*}
&d_{r,\Aff}(\zeta_1,\zeta_2)=\\
&=\max\left\{\left|\frac{\xi_1(0)}{\eta_1(0)}-\frac{\xi_2(0)}{\eta_2(0)}\right|,\big\|\eta_1\circ L_{\eta_1}-\eta_2 \circ L_{\eta_2}\big\|_{C^r([0,1])},\big\|\xi_1\circ L_{\xi_1}-\xi_2 \circ L_{\xi_2}\big\|_{C^r([-1,0])}\right\}.\\
\end{align*}

Both $d_r$ and $d_{r,\Aff}$ are not metrics but \emph{pseudo-metrics}, since they are invariant under conjugacy by homotheties: if $\alpha$ is a positive real number, $H_{\alpha}(t)=\alpha t$ and $\zeta_1=H_{\alpha} \circ \zeta_2 \circ H_{\alpha}^{-1}$, then $d_r(\zeta_1,\zeta_2)=d_{r,\Aff}(\zeta_1,\zeta_2)=0$. In order to have metrics we need to restrict to \emph{normalized} critical commuting pairs (recall that for a commuting pair $\zeta=(\eta,\xi)$ we denote by $\tilde{\zeta}$ the pair $(\tilde{\eta}|_{\tilde{I_{\eta}}}, \tilde{\xi}|_{\tilde{I_{\xi}}})$, where tilde means rescaling by the linear factor $\lambda=\frac{1}{|I_{\xi}|}$).

\subsection{The renormalization operator} As we just said, for a commuting pair $\zeta=(\eta,\xi)$ we denote by $\tilde{\zeta}$ the pair $(\tilde{\eta}|_{\tilde{I_{\eta}}}, \tilde{\xi}|_{\tilde{I_{\xi}}})$, where tilde means rescaling by the linear factor $\lambda=\frac{1}{|I_{\xi}|}$. Note that $|\tilde{I_{\xi}}|=1$ and $\tilde{I_{\eta}}$ has length equal to the ratio between the lengths of $I_{\eta}$ and $I_{\xi}$. Equivalently $\tilde{\eta}(0)=-1$ and $\tilde{\xi}(0)=\frac{|I_{\eta}|}{|I_{\xi}|}=\xi(0)/\big|\eta(0)\big|$.

Let $\zeta=(\eta,\xi)$ be a $C^r$ critical commuting pair according to Definition \ref{critpair}, and recall that $\big(\eta\circ\xi\big)(0)=\big(\xi\circ\eta\big)(0) \neq 0$. Let us suppose that $\big(\xi\circ\eta\big)(0) \in I_{\eta}$ (just as in both Figure 1 and Figure 2 above) and define the \emph{period} $\chi(\zeta)$ of the commuting pair $\zeta=(\eta,\xi)$ as $a\in\nt$ if:$$\eta^{a+1}\big(\xi(0)\big) < 0 \leq \eta^a\big(\xi(0)\big)$$and $\chi(\zeta)=\infty$ if no such $a$ exists (note that in this case the map $\eta|_{I_\eta}$ has a fixed point, so when we are dealing with commuting pairs induced by critical circle maps with irrational rotation number we have finite period). Note also that the period of the pair $(f^{q_{n+1}}|_{I_n},f^{q_n}|_{I_{n+1}})$ induced by a critical circle maps $f$ is exactly $a_{n+1}$, where $\rho(f)=[a_0,a_1,a_2,...,a_n,a_{n+1},...]$ (because the combinatorics of $f$ are the same as for the corresponding rigid rotation). For a pair $\zeta=(\eta,\xi)$ with $\big(\xi\circ\eta\big)(0) \in I_{\eta}$ and $\chi(\zeta)=a<\infty$ note that the pair:$$\big(\eta|_{[0,\eta^a(\xi(0))]}\,,\,\eta^a \circ \xi|_{I_\xi}\big)$$is again a commuting pair, and if $\zeta=(\eta,\xi)$ is induced by a critical circle map:$$\zeta=(\eta,\xi)=\big(f^{q_{n+1}}|_{I_n}\,,\,f^{q_n}|_{I_{n+1}}\big)$$we have that:$$\big(\eta|_{[0,\eta^a(\xi(0))]}\,,\,\eta^a \circ \xi|_{I_\xi}\big)=\big(f^{q_{n+1}}|_{I_{n+2}}\,,\,f^{q_{n+2}}|_{I_{n+1}}\big)$$

This motivates the following definition (the definition in the case $\big(\xi\circ\eta\big)(0) \in I_{\xi}$ is analogous):

\begin{definition}\label{renop} Let $\zeta=(\eta,\xi)$ be a critical commuting pair with $\big(\xi\circ\eta\big)(0) \in I_{\eta}$. We say that $\zeta$ is \emph{renormalizable} if $\chi(\zeta)=a<\infty$. In this case we define the \emph{pre-renormalization} of $\zeta$ as the critical commuting pair:$$p\mathcal{R}(\zeta)=\big(\eta|_{[0,\eta^a(\xi(0))]}\,,\,\eta^a \circ \xi|_{I_\xi}\big),$$and we define the \emph{renormalization} of $\zeta$ as the normalization of $p\mathcal{R}(\zeta)$, that is:$$\mathcal{R}(\zeta)=\tilde{p\mathcal{R}(\zeta)}=\left(\tilde{\eta}|_{\tilde{[0,\eta^a(\xi(0))]}}\,,\,\tilde{\eta^a \circ \xi}|_{\tilde{I_\xi}}\right).$$
\end{definition}

Note in particular that $d_r\big(p\mathcal{R}(\zeta_0),p\mathcal{R}(\zeta_1)\big)=d_r\big(\mathcal{R}(\zeta_0),\mathcal{R}(\zeta_1)\big)$ for any two critical commuting pairs $\zeta_0$ and $\zeta_1$ renormalizable with the same period.

A critical commuting pair is a special case of a \emph{generalized interval exchange map} of two intervals, and the renormalization operator defined above is just the restriction of the \emph{Zorich accelerated version} of the \emph{Rauzy-Veech renormalization} for interval exchange maps. However we will keep in this article the classical terminology for critical commuting pairs.

\begin{figure}[ht!]
\begin{center}
\includegraphics[scale=1.0]{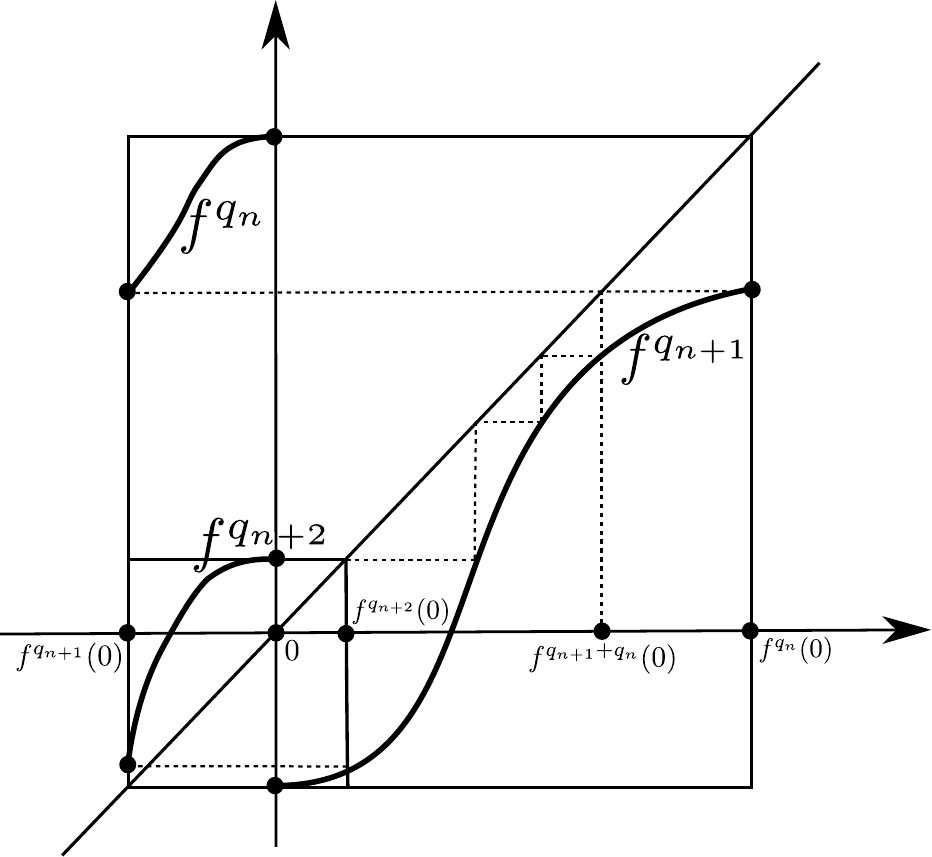}
\caption{Two consecutive renormalizations of $f$, without rescaling (recall that $f^{q_n}$ means $T^{-p_n}\circ\widehat{f}^{q_n}$, see Section \ref{subccp}). In this example $a_{n+1}=4$.}
\end{center}
\end{figure}

\begin{definition}\label{rotnum} Let $\zeta$ be a critical commuting pair. If $\chi(\mathcal{R}^j(\zeta))<\infty$ for $j \in \{0,1,...,n-1\}$ we say that $\zeta$ is \emph{n-times renormalizable}, and if $\chi(\mathcal{R}^j(\zeta))<\infty$ for all $j\in\nt$ we say that $\zeta$ is \emph{infinitely renormalizable}. In this case the irrational number $\theta$ whose continued fraction expansion is equal to:$$\big[\chi\big(\zeta\big), \chi\big(\mathcal{R}(\zeta)\big),...,\chi\big(\mathcal{R}^n(\zeta)\big),\chi\big(\mathcal{R}^{n+1}(\zeta)\big),... \big]$$is called the \emph{rotation number} of the critical commuting pair $\zeta$, and denoted by $\rho(\zeta)=\theta$.
\end{definition}

An immediate remark is that when $\zeta$ is induced by a critical circle map with irrational rotation number, the pair $\zeta$ is automatically infinitely renormalizable and both definitions of rotation number coincide: any $C^r$ critical circle map $f$ with irrational rotation number gives rise to the orbit $\big\{\mathcal{R}^n(f)\big\}_{n \geq 1}$ of infinitely renormalizable $C^r$ critical commuting pairs defined by:$$\mathcal{R}^n(f)=\left(\tilde{f}^{q_n}|_{\tilde{I}_{n-1}},\tilde{f}^{q_{n-1}}|_{\tilde{I}_n}\right)\quad\mbox{for all $n \geq 1$.}$$

For any positive number $\theta$ denote by $\lfloor\theta\rfloor$ the \emph{integer part} of $\theta$, that is, $\lfloor\theta\rfloor\in\nt$ and $\lfloor\theta\rfloor\leq\theta<\lfloor\theta\rfloor+1$. Recall that the \emph{Gauss map} $G:[0,1]\to[0,1]$ is defined by:$$G(\theta)=\frac{1}{\theta}-\left\lfloor\frac{1}{\theta}\right\rfloor\quad\mbox{for}\quad\theta\neq 0\quad\mbox{and}\quad G(0)=0\,,$$and note that $\rho$ \emph{semi-conjugates} the renormalization operator with the Gauss map:$$\rho\big(\mathcal{R}^n(\zeta)\big)=G^n\big(\rho(f)\big)$$for any $\zeta$ at least $n$-times renormalizable. In particular the renormalization operator acts as a \emph{left shift} on the continued fraction expansion of the rotation number: if $\rho(\zeta)=[a_0,a_1,...]$ then $\rho\big(\mathcal{R}^n(\zeta)\big)=[a_n,a_{n+1},...]$.

\subsection{Renormalization of real analytic critical commuting pairs} The main result on the dynamics of the renormalization operator acting on the real analytic category that we will need in this paper is the following:

\begin{theorem}[de Faria-de Melo 2000, Khmelev-Yampolsky 2006]\label{uniform} There exists a universal constant $\lambda$ in $(0,1)$ with the following property: given two real analytic critical commuting pairs $\zeta_1$ and $\zeta_2$ with the same irrational rotation number and the same odd type at the critical point, there exists a constant $C>0$ such that:$$d_{r}\big(\mathcal{R}^n(\zeta_1),\mathcal{R}^n(\zeta_2)\big) \leq C\lambda^n$$for all $n\in\nt$ and for any $0 \leq r < \infty$. Moreover given a $C^{\omega}$-compact set $\mathcal{K}$ of real analytic critical commuting pairs, the constant $C$ can be chosen the same for any $\zeta_1$ and $\zeta_2$ in $\mathcal{K}$.
\end{theorem}

As we said in Section \ref{S:sketch}, Theorem \ref{uniform} was proved by de Faria and de Melo \cite{edsonwelington2} for rotation numbers of bounded type, and extended by Khmelev and Yampolsky \cite{KY} to cover all irrational rotation numbers.

\subsection{Technical tools} We finish this preliminary section by stating some well-know distortion estimates that we will use along the text. For intervals $M \subset T$ we define the \emph{space} of $M$ inside $T$ to be the smallest of the ratios $|L|/|M|$ and $|R|/|M|$, where $L$ and $R$ are the left and right components of $T\!\setminus\!M$.

\begin{theorem}[Koebe's distortion principle for real maps]\label{koebedistint} Let $f:S^1 \to S^1$ be a $C^3$ map, and let $S>0$ and $\tau>0$ be two positive constants. Then there exists a constant $K=K(S,\tau,f)>1$ with the following property: if $T$ is an interval in the unit circle such that $f^m|_{T}$ is a diffeomorphism and if $\sum_{j=0}^{m-1}\big|f^j(T)\big| \leq S$, then for each interval $M \subset T$ for which the space of $f^m(M)$ inside $f^m(T)$ is at least $\tau$ and for all $x,y \in M$ we have that:$$\frac{1}{K}\leq\frac{\big|Df^m(x)\big|}{\big|Df^m(y)\big|}\leq K\,.$$
\end{theorem}

For a proof of Theorem \ref{koebedistint} see \cite[Section IV.3, Theorem 3.1, page 295]{dmvs}. In Section \ref{melhoratese} we will also need the following two classical results from complex analysis:

\begin{theorem}[Koebe's distortion principle for complex maps]\label{koebehol} If $K$ is a compact subset of a domain $\Omega\subset\C$ there exists a constant $M$ (depending on $K$) such that for every univalent function $f$
on $\Omega$ and every pair of points $z,w \in K$ we have
\begin{equation*}
\frac{1}{M}
\leq
\frac{\lvert Df(z) \rvert}{\lvert Df(w) \rvert}
\leq
M\,.
\end{equation*} 
\end{theorem}

\begin{theorem}[Koebe's one-quarter theorem]\label{oneq} Let $f:\D\to\C$ be a univalent map from the open unit disk into the complex plane such that the origin is an indifferent fixed point, that is, $f(0)=0$ and $\big|Df(0)\big|=1$. Then the image of $f$ contains the open disk of radius $1/4$ around the origin.
\end{theorem}

\section{Controlled pairs and real bounds}\label{Seccont}

\begin{definition}\label{negschar} We say that a $C^3$ critical commuting pair $\zeta=(\eta,\xi)$ has \emph{negative Schwarzian} if the Schwarzian derivative of both $\eta$ and $\xi$ are negative in $I_{\eta}\!\setminus\!\{0\}$ and $I_{\xi}\!\setminus\!\{0\}$ respectively.
\end{definition}

Recall that the \emph{Schwarzian derivative} of a $C^3$ map $f$ at a regular point $x$ is defined as:$$Sf(x)=\frac{D^3f(x)}{Df(x)}-\frac{3}{2}\left(\frac{D^2f(x)}{Df(x)}\right)^2.$$

\begin{remark}\label{chainsch} Let $\zeta$ be a $C^3$ critical commuting pair with negative Schwarzian. If $\zeta$ is renormalizable, then $\mathcal{R}(\zeta)$ has negative Schwarzian (this follows at once from the chain rule for the Schwarzian derivative). In particular if $\zeta$ has negative Schwarzian and is infinitely renormalizable, then $\mathcal{R}^n(\zeta)$ has negative Schwarzian for all $n\in\nt$.
\end{remark}

\begin{definition}\label{defcont} Let $K>1$ and let $\zeta=(\eta,\xi)$ be a normalized $C^3$ critical commuting pair which is renormalizable with some period $a\in\nt$. We say that $\zeta$ is \emph{$K$-controlled} if the following seven conditions are satisfied:
\begin{itemize}
\item $1/K \leq \xi(0) \leq K$.
\item $\xi(0)-\eta\big(\xi(0)\big) \geq 1/K$.
\item $\eta^{a-1}\big(\xi(0)\big)-\eta^{a}\big(\xi(0)\big) \geq 1/K$.
\item $\eta^{a}\big(\xi(0)\big) \geq 1/K$.
\item $\eta^{a+1}\big(\xi(0)\big) \leq -1/K$.
\item $\|\xi\|_{C^3([-1,0])}\leq K$ and $\|\eta\|_{C^3([0,\xi(0)])}\leq K$.
\item $D\eta(x)\geq 1/K$ for all $x\in\big[\eta^{a}(\xi(0)),\xi(0)\big]$.
\end{itemize}
\end{definition}

Of course if $\zeta$ is $K_0$-controlled and $K_1 \geq K_0$, then $\zeta$ is also $K_1$-controlled.

\begin{definition}\label{defdelconjK}
For $K>1$ let $\mathcal{K}=\mathcal{K}(K)$ be the space of normalized $C^3$ critical commuting pairs which are $K$-controlled. For $K>1$ and $a\in\nt$ let $\mathcal{K}_a(K)$ be the space of normalized $C^3$ critical commuting pairs which are renormalizable with period $a$ and $K$-controlled.
\end{definition}

We can restate the real bounds (Theorem \ref{realbounds}) in the following way:

\begin{theorem}[Real bounds]\label{realB} There exists a universal constant $K_0>1$ with the following property: for any given $C^4$ critical circle map $f$ with irrational rotation number there exists $n_0=n_0(f)\in\nt$ such that the critical commuting pair $\mathcal{R}^n(f)$ is $K_0$-controlled for any $n \geq n_0$.
\end{theorem}

The $C^4$ smoothness is needed in order to have that the critical commuting pair $\mathcal{R}^n(f)$ is $C^3$ bounded for $n$ big enough. For a proof of Theorem \ref{realB} see \cite[Section 3 and Appendix A]{dfdm1}. Moreover, we have:

\begin{theorem}\label{realBcomSneg} For any given $K>1$ there exists $n_0=n_0(K)\in\nt$ with the following property: if $\zeta$ is an infinitely renormalizable $C^3$ critical commuting pair with negative Schwarzian which is $K$-controlled, then $\mathcal{R}^n(\zeta)$ is $K_0$-controlled for all $n \geq n_0$, where the universal constant $K_0>1$ is given by Theorem \ref{realB}.
\end{theorem}

We will also need the following fact:

\begin{lema}\label{fromC2toC3} Given $K>1$ there exists $B=B(K)>K$ with the following property: let $\zeta=(\eta,\xi)$ be a $C^3$ critical commuting pair which is $K$-controlled, renormalizable with some period $a\in\nt$ and such that $\mathcal{R}(\zeta)$ is $C^2$-bounded by $K$. Then $\mathcal{R}(\zeta)$ is $C^3$-bounded by $B$.
\end{lema}

Note that the constant $B$ depends only on $K$, and not on the period of renormalization $a$.

\begin{proof}[Proof of Lemma \ref{fromC2toC3}] Let $\zeta=(\eta,\xi)$ be a $C^3$ normalized critical commuting pair which is renormalizable with some period $a\in\nt$. For $i\in\{0,...,a\}$ let $x_i=\eta^{i}(\xi(0))$. Note that $x_i\in I_{\eta}=[0,\xi(0)]$ for all $i\in\{0,...,a\}$. Denote by $I_i$, $i\in\{1,...,a\}$, the fundamental domains of $\eta$ given by $I_i=[\eta^{i}(\xi(0)),\eta^{i-1}(\xi(0))]$. By the commuting condition $I_1=\xi(I_{\xi})=\xi ([-1,0])$. As before, the letter $B$ will denote uniform constants, their value might change during estimates. We claim first that:
\begin{equation}\label{claimsch}
\big|S\eta^a(x)\big| \leq B\quad\mbox{for all $x \in I_1$,}
\end{equation}
where as before $S\eta^a$ denotes the Schwarzian derivative of $\eta^a$. Indeed, by the $K$-control we have $\big|S\eta(y)\big| \leq B$ for all $y\in[x_a,x_0]=\big[\eta^a(\xi(0)),\xi(0)\big]$ and then:$$\big|S\eta^a(x)\big|\leq\sum_{i=0}^{a-1}\big|S\eta(\eta^i(x))\big|\!\times\!\big|D\eta^i(x)\big|^2\leq B\sum_{i=0}^{a-1}\big|D\eta^i(x)\big|^2\quad\mbox{for all $x \in I_1$.}$$

By bounded distortion (Theorem \ref{koebedistint}) and the $K$-control we have that:$$\big|D\eta^i(x)\big|\leq \frac{B}{|I_1|}\!\times\!|I_{i+1}|\leq B|I_{i+1}|\quad\mbox{for all $x \in I_1$,}$$and then we obtain that for all $x \in I_1$:
\begin{align*}
\big|S\eta^a(x)\big|&\leq B\sum_{i=0}^{a-1}|I_{i+1}|^2=B\sum_{i=1}^{i=a}|I_{i}|^2\\
&\leq B\times\!\max_{i\in\{1,...,a\}}\!|I_{i}|\times\!\sum_{i=1}^{i=a}|I_{i}|\leq B\big|\xi(0)\big|^2 \leq B
\end{align*}
as was claimed. Note now that by hypothesis we know that:$$\big|D\eta^a(x)\big| \leq K\quad\mbox{and}\quad\big|D^2\eta^a(x)\big|\leq K\quad\mbox{for all $x \in I_1$.}$$

Moreover, again by bounded distortion and the $K$-control we also have $\big|D\eta^a(x)\big| \geq 1/B$ for all $x \in I_1$, since $|I_1|$ and $|I_a|$ are comparable. With this at hand and claim \eqref{claimsch} we obtain for all $x \in I_1$ that:$$\big|D^3\eta^a(x)\big|\leq\big|D\eta^a(x)\big|\!\times\!\big|S\eta^a(x)\big|+\frac{3}{2}\,\frac{\big|D^2\eta^a(x)\big|^2}{\big|D\eta^a(x)\big|}\leq BK+\frac{3}{2}\,BK^2\,.$$

Since $\|\eta^a\|_{C^2(I_1)} \leq K$ by hypothesis, we obtain that $\|\eta^a\|_{C^3(I_1)} \leq B$. Finally, from $p\mathcal{R}(\zeta)=\big(\eta|_{[0,\eta^a(\xi(0))]},\eta^a \circ \xi|_{I_\xi}\big)$ and the fact that $\zeta$ is $K$-controlled we obtain that the critical commuting pair $\mathcal{R}(\zeta)$ is $C^3$-bounded by $B$.
\end{proof}

\begin{remark}\label{rememap} Let $\zeta=(\eta,\xi)$ be a $C^3$ critical commuting pair which is renormalizable with period $a$ and $K_0$-controlled, and denote by $N\eta$ the \emph{non-linearity} of $\eta$ (see Section \ref{comp}). Then $N\eta$ is Lipschitz continuous in $\big[\eta^{a}(\xi(0)),\xi(0)\big]$ with some universal constant $K(K_0)>1$. Indeed, this comes from the identity:$$D(N\eta)=\frac{D^3\eta}{D\eta}-\left(\frac{D^2\eta}{D\eta}\right)^2=S\eta+\frac{1}{2}\big(N\eta\big)^2,$$where as before $S\eta$ denotes the Schwarzian derivative of $\eta$. Note that actually $N\eta$ is $C^1$-bounded in $\big[\eta^{a}(\xi(0)),\xi(0)\big]$ whenever $\zeta\in\mathcal{K}_a$. Moreover, if $\zeta,\tilde{\zeta}\in\mathcal{K}_a$ and $x \in J=\big[\eta^{a}(\xi(0)),\xi(0)\big]\cap\big[\tilde{\eta}^{a}(\tilde{\xi}(0)),\tilde{\xi}(0)\big]$ we have that $\big|N\eta(x)-N\tilde{\eta}(x)\big|\leq K\|\eta-\tilde{\eta}\|_{C^2(J,\R)}$.
\end{remark}

Finally, it is not difficult to prove that given $K_0>1$ there exists $K=K(K_0)>1$ such that both metrics $d_2$ and $d_{2,\Aff}$ are Lipschitz equivalent, with constant $K$, when restricted to normalized $K_0$-controlled $C^3$ critical commuting pairs. This allows us to use both metrics on our estimates.

\section{Lipschitz continuity}\label{Ssap}

Sections \ref{Smonfam} to \ref{lip} of this paper are devoted to prove the following:

\begin{lema}[Key lemma]\label{main} For any given $K>1$ there exist two constants $\varepsilon_0=\varepsilon_0(K)\in(0,1)$ and $L=L(K)>1$ with the following property: let $\zeta_0$ and $\zeta_1$ be two infinitely renormalizable normalized $C^3$ critical commuting pairs which are $K$-controlled, both $\zeta_0$ and $\zeta_1$ have negative Schwarzian, $\rho(\zeta_0)=\rho(\zeta_1)\in[0,1]\!\setminus\!\Q$ and $d_2(\zeta_0,\zeta_1)<\varepsilon_0$. Then we have:$$d_2\big(\mathcal{R}(\zeta_0),\mathcal{R}(\zeta_1)\big) \leq L\,d_2(\zeta_0,\zeta_1),$$where $d_2$ denotes the $C^2$ distance in the space of $C^2$ critical commuting pairs.
\end{lema}

We remark that the Lipschitz constant $L$ given by Lemma \ref{main} depends only on $K$, it do not depends on the common combinatorics of the critical pairs $\zeta_0$ and $\zeta_1$.

Let $\zeta=(\eta,\xi)$ be a $C^3$ $K$-controlled critical commuting pair which is renormalizable with some period $a\in\nt$. As before (see the proof of Lemma \ref{fromC2toC3} in Section \ref{Seccont}) we will use the following notation: for $i\in\{0,...,a\}$ let $x_i=\eta^{i}(\xi(0))$. Note that $x_i\in I_{\eta}=[0,\xi(0)]$ for all $i\in\{0,...,a\}$. Denote by $I_i$, $i\in\{1,...,a\}$, the fundamental domains of $\eta$ given by $I_i=[\eta^{i}(\xi(0)),\eta^{i-1}(\xi(0))]$. By the commuting condition $I_1=\xi(I_{\xi})=\xi ([-1,0])$. The following result due to J.-C. Yoccoz plays a fundamental role in our analysis:

\begin{lema}[Yoccoz's Lemma]\label{aquele} Assume that $\zeta$ has negative Schwarzian, and let $N\in\{1,...,a\}$ defined by $x_{N+1} \leq p \leq x_N$. Then we have:
\begin{equation}\label{estyoc}
|I_i|\asymp\frac{1}{i^2}\quad\mbox{for $i\in\{1,...,N\}$,}\quad\mbox{and}\quad|I_i|\asymp\frac{1}{(a-i)^2}\quad\mbox{for $i\in\{N,...,a-1\}$.}
\end{equation}

Moreover:
\begin{itemize}
\item $N \asymp a$, that is, there exist two constants $\delta_0=\delta_0(\mathcal{K})$ and $\delta_1=\delta_1(\mathcal{K})$ with $0<\delta_0\leq\delta_1<1$ such that $\delta_0\,a \leq N \leq \delta_1a$.
\item $\big|x_i-p\big|\asymp\sum_{j=i+1}^{j=N}\frac{1}{j^2}\asymp\frac{1}{i+1}$ for all $i\in\big\{1,...,\lfloor\frac{a}{2}\rfloor\big\}$.
\end{itemize}
\end{lema}

Since $N \asymp a$ we can restate \eqref{estyoc} in the following way (see \cite[Section 4.1, page 354]{dfdm1}):

\begin{lem}[Yoccoz's Lemma]\label{doyoc} Assume that $\zeta$ has negative Schwarzian. There exists $K=K(\mathcal{K})>1$ such that for all $i\in\{1,...,a\}$ we have:$$\frac{1}{K}\frac{1}{\min\{i,a-i\}^2}\leq|I_i| \leq K\frac{1}{\min\{i,a-i\}^2}\,.$$
\end{lem}

For a proof of Yoccoz's Lemma see \cite[Appendix B, page 386]{dfdm1}.

\section{Standard families}\label{Smonfam}

Fix $K_0>1$ and let $\mathcal{K}$ be the space of normalized $C^3$ critical commuting pairs which are $K_0$-controlled (see Definition \ref{defcont} in Section \ref{Seccont}). We will consider in this section a $C^3$ critical commuting pair $\zeta=(\eta,\xi)$ with negative Schwarzian that belongs to $\mathcal{K}$ which is renormalizable with period $a\in\nt$. For such a pair we will construct/define the corresponding \emph{standard family}.

\subsection{Glueing procedure and translations} In the notation of the proof of Lemma \ref{laP4} in Section \ref{Sprel} we have:

\begin{lema}\label{folga} There exists $s_0=s_0(\mathcal{K})>0$ such that for any $\zeta=(\eta,\xi)\in\mathcal{K}$ both components of $A\setminus\big\{\eta(0)\big\}$ and both components of $B\setminus\big\{\xi(0)\big\}$ have Euclidean length greater than or equal to $s_0$.
\end{lema}

\begin{proof}[Proof of Lemma \ref{folga}] There exist positive constants $\delta$ and $\rho$ (depending only on $K_0$) such that both components of $C\!\setminus\!\{0\}$ have Euclidean length greater than or equal to $\delta$, $\inf_{C}\{D\phi\}>\rho$ and $\inf_{C}\{D\psi\}>\rho$. Then it is enough to take $0<s_0<(\delta\rho)^{2d+1}$, where the integer $2d+1$ is the criticality of $\eta$ and $\xi$ at the origin (See Condition \eqref{pc} in Definition \ref{critpair}).
\end{proof}

Still in the notation of the proof of Lemma \ref{laP4} let $M=V_{-} \cup V_{+}/\sim$ where $x \sim y$ if $x \in A$, $y \in B$ and $\widehat{\xi}(x)=\widehat{\eta}(y)$. Note that $\eta(0)\sim\xi(0)$ by the commuting condition \eqref{comor} in Definition \ref{critpair}. Let $p:V_{-} \cup V_{+} \to M$ be the canonical projection for the identification $\sim$, and note that $M$ is a compact boundaryless one-dimensional $C^3$ manifold since the map $\widehat{\eta}^{-1}\circ\widehat{\xi}:A \to B$ is a $C^3$ diffeomorphism (it can be proved that $p$ is the restriction of a $C^3$ covering map from the real line to $M$, but this fact will not be needed in this paper).

\begin{lema}\label{MFcontrolP} There exists a $C^3$ diffeomorphism $\psi:M \to S^1$ such that defining $P:V_{-} \cup V_{+} \to S^1$ as $P=\psi \circ p$ we have that for all $x,y \in A \cap I_{\xi}$, for all $x,y \in B \cap I_{\eta}$ and for all $x,y \in (I_{\xi}\cup I_{\eta})\!\setminus\!(A \cup B)$:$$\frac{|x-y|}{K}\leq d\big(P(x),P(y)\big)\leq K|x-y|$$for some universal constant $K=K(\mathcal{K})>1$, where $d$ denotes the Euclidean distance in the unit circle.
\end{lema}

From now on let $P:V_{-} \cup V_{+} \to S^1$ be the $C^3$ map defined in Lemma \ref{MFcontrolP}. Given $t\in\R$ we define the \emph{translation} by $t$ on $I_{\xi} \cup I_{\eta}$ to be the $C^3$ map $T:I_{\xi} \cup I_{\eta}\times\R \to I_{\xi} \cup I_{\eta}$ given by:$$\big(P \circ T_t\big)(x)=e^{2\pi it}P(x)\,,$$that is, $T(x,t)=T_t(x)=P^{-1}\big(e^{2\pi it}P(x)\big)$, whenever is clear which preimage under $P$ we choose for points in $P(A)$. In particular $T_0$ is the identity on $I_{\xi} \cup I_{\eta}$. Note also that:$$\frac{\partial T}{\partial t}(x,t)=\frac{1}{DP\big(T_t(x)\big)}\quad\mbox{and}\quad\frac{\partial T}{\partial x}(x,t)=\frac{DP(x)}{DP\big(T_t(x)\big)}\,,$$and from Lemma \ref{MFcontrolP} we get that $\frac{1}{K}\leq\frac{\partial T}{\partial t}(x,t) \leq K$ for all $x\in I_{\xi} \cup I_{\eta}$.

\subsection{Standard families of commuting pairs}\label{subMF} By Condition \eqref{la5} in Definition \ref{critpair} the discontinuous piecewise smooth map $\tilde{f}_{\zeta}:I_{\xi} \cup I_{\eta} \to I_{\xi} \cup I_{\eta}$ given by:
$$\tilde{f}_{\zeta}(x)=\left\{\begin{array}{ll}
\xi(x)&\mbox{for } x \in I_{\xi}\\
\eta(x)&\mbox{for } x \in I_{\eta}\\
\end{array}\right.$$projects under $p$ to a $C^3$ homeomorphism of the quotient manifold $M$, and then it projects under $P$ to a $C^3$ critical circle map $f_{\zeta}$ in $S^1$.

By Lemma \ref{folga} and Lemma \ref{MFcontrolP} above, the Euclidean length of both components of $P(A)\!\setminus\!\big\{f_{\zeta}\big(P(0)\big)\big\}$ in $S^1$ is bounded from below by some positive constant $l_0$, universal in $\mathcal{K}$. For $t \in W=(-l_0,l_0)$ let $f_{t}:S^1 \to S^1$ be the $C^3$ critical circle map given by $f_{t}(z)=e^{2\pi it}f_{\zeta}(z)$, and note that $f_0=f_{\zeta}$. Since the critical value of $f_{t}$ (which is $e^{2\pi it}f_{\zeta}\big(P(0)\big)$) belongs to $P(A)$ we can lift each $f_{t}$ up to a $C^3$ critical commuting pair $\zeta_{t}=(\eta_t,\xi_t)$ with:$$\xi_t(x)=\big(T_t\circ\xi_0\big)(x)=T\big(\xi_0(x),t\big)\quad\mbox{and}\quad\eta_t(x)=\big(T_t\circ\eta_0\big)(x)=T\big(\eta_0(x),t\big).$$

Note that:$$\frac{\partial\xi_t}{\partial t}(x)=\frac{1}{DP\big(\xi_t(x)\big)}\quad\mbox{and}\quad\frac{\partial\eta_t}{\partial t}(x)=\frac{1}{DP\big(\eta_t(x)\big)}\,.$$

\begin{lema}\label{t0mon} There exists $K(\mathcal{K})>1$ such that $|t|/K \leq d_2(\zeta_0,\zeta_t) \leq K|t|$ for all $t \in W$.
\end{lema}

Now let $W_a \subset W$ be the set of all $t \in W$ such that $\zeta_t$ is renormalizable with period $a$, that is:$$W_a=\left\{t \in W:\left\lfloor\frac{1}{\rho(\zeta_t)}\right\rfloor=\left\lfloor\frac{1}{\rho(\zeta_0)}\right\rfloor=a\right\}.$$

\begin{lema}\label{Wa} There exists $a_0(\mathcal{K})\in\nt$ such that if $a \geq a_0$ we have that $\overline{W_a} \subset W$. If we denote the boundary points of $W_a$ by $-w_{-}^{a}$ and $w_{+}^{a}$, that is, $W_a=[-w_{-}^{a},w_{+}^{a}]$, we have that:$$\eta_{-w_{-}^{a}}^{a+1}\big(\xi_{-w_{-}^{a}}(0)\big)=0\quad\mbox{and}\quad\eta_{w_{+}^{a}}^{a}\big(\xi_{w_{+}^{a}}(0)\big)=0\,.$$
\end{lema}

\begin{proof}[Proof of Lemma \ref{Wa}] By Lemma \ref{MFcontrolP} there exists a universal upper bound $K>0$ for the first derivative of $P$ in $V_{-} \cup V_{+}$. By Yoccoz's Lemma (Lemma \ref{doyoc}) it is enough to take $a_0\gtrsim\big(K/|W|\big)^{1/2}$ in order to have $|W| \gtrsim K/a_0^2$. The assertion about the boundary of $W_a$ follows by combinatorics.
\end{proof}

\begin{coro}\label{corodelWa} Let $a_0=a_0(\mathcal{K})\in\nt$ given by Lemma \ref{Wa}. Let $\zeta$ be a normalized $C^3$ critical commuting pairs that belongs to $\mathcal{K}$ which is renormalizable with period $a \geq a_0$. Given $x\in\left[0,\eta^a\big(\xi(0)\big)\right]$ there exists $t_x \leq 0$ in $W_a(\zeta)$ such that $\eta_{t_x}^{a}\big(\xi_{t_x}(0)\big)=x$.
\end{coro}

Finally, let $V=[-v_-,v_+] \subset W_a$ defined by:$$\eta_{-v_-}^{a+1}\big(\xi_{-v_-}(0)\big)=-1/K_0^2\quad\mbox{and}\quad\eta_{v_+}^{a}\big(\xi_{v_+}(0)\big)=1/K_0^2\,.$$

\begin{lema}\label{BDinV} For any $t \in V$ and any $k\in\{1,...,a-1\}$ the $C^3$ diffeomorphism $\eta_t^{a-k}:I_k(t) \to I_a(t)$ has universally bounded distortion.
\end{lema}

Here $I_i(t)=\big[x_i(t),x_{i-1}(t)\big]$ for all $i\in\{1,...,a\}$.

\begin{proof}[Proof of Lemma \ref{BDinV}] Combine Koebe's distortion principle (Theorem \ref{koebedistint}) with the $K$-control.
\end{proof}

\begin{lema}\label{temamigo} Let $a_0=a_0(\mathcal{K})\in\nt$ given by Lemma \ref{Wa}. Let $\zeta_0=(\eta_0,\xi_0)$ and $\zeta_1=(\eta_1,\xi_1)$ be two normalized $C^3$ critical commuting pairs that belong to $\mathcal{K}$ which are renormalizable with the same period $a \geq a_0$. Then there exists $t_0 \in V(\zeta_0) \subset W_a(\zeta_0)$ such that:$$\eta_{t_0}^a\big(\xi_{t_0}(0)\big)=\eta_1^a\big(\xi_1(0)\big)\quad\mbox{and}\quad d_2(\zeta_0,\zeta_{t_0}) \leq Kd_2(\zeta_0,\zeta_1),$$where the constant $K=K(\mathcal{K})>1$ is given by Lemma \ref{t0mon}.
\end{lema}

\begin{figure}[h]
\begin{center}
\includegraphics[scale=1.1]{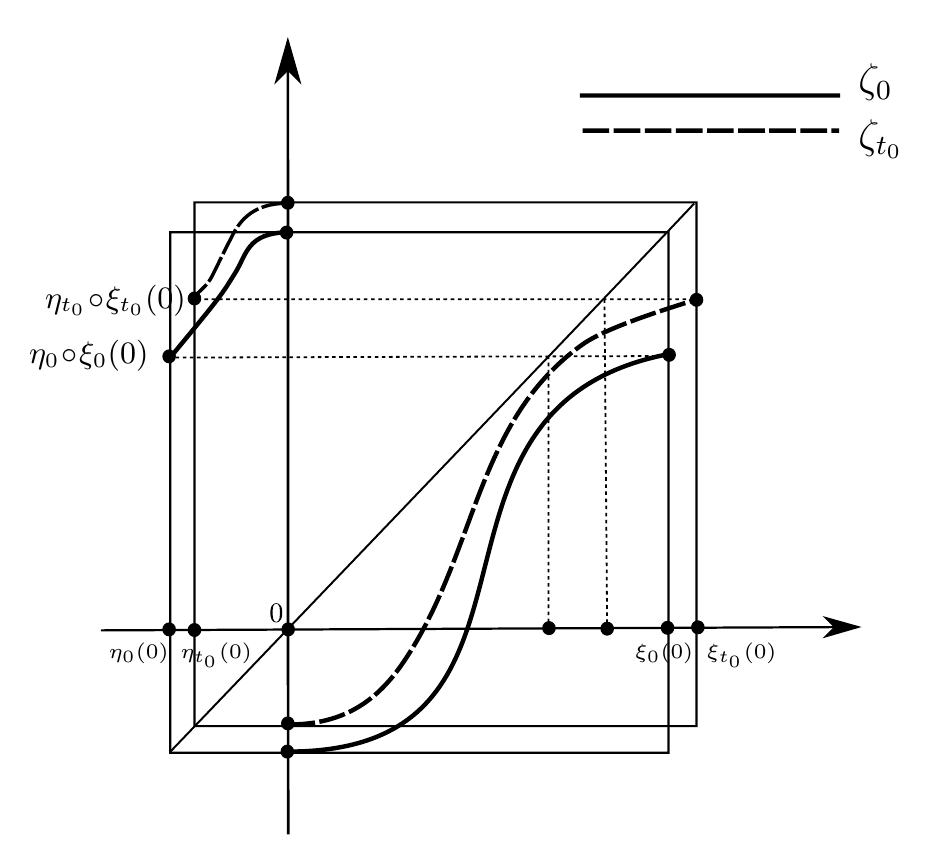}
\caption{Standard families of critical commuting pairs (in this figure, the period of $\zeta_0$ is equal to $3$, while the period of $\zeta_{t_0}$ is $8$).}
\end{center}
\end{figure}

\begin{proof}[Proof of Lemma \ref{temamigo}] We may suppose that $\eta_{0}^{a}\big(\xi_0(0)\big)\geq\eta_{1}^{a}\big(\xi_1(0)\big)$, that is, $\eta_{1}^{a}\big(\xi_1(0)\big)$ belongs to the interval $\big[1/K_0,\eta_{0}^{a}\big(\xi_0(0)\big)\big]\subset[1/K_0,K_0]$. By Corollary \ref{corodelWa} there exists $t_0<0$ in $V(\zeta_0)$ such that $\eta_{t_0}^a\big(\xi_{t_0}(0)\big)=\eta_1^a\big(\xi_1(0)\big)$. Note that $\eta_{t_0}^{a+1}\big(\xi_{t_0}(0)\big)\leq\eta_0^{a+1}\big(\xi_0(0)\big)\leq-1/K_0<-1/K_0^2$. Now let $K=K(\mathcal{K})>1$ given by Lemma \ref{t0mon}. We claim that $|t_0| \leq Kd_2(\zeta_0,\zeta_1)$. Indeed, if $|t_0|>Kd_2(\zeta_0,\zeta_1)$ we would have $\xi_{t_0}<\xi_1$ and $\eta_{t_0}<\eta_1$ in the corresponding intersections of domains, but this implies that $\eta_{t_0}^a\big(\xi_{t_0}(0)\big)<\eta_1^a\big(\xi_1(0)\big)$ which is a contradiction. Then $|t_0| \leq Kd_2(\zeta_0,\zeta_1)$ and we are done.
\end{proof}

\subsection{Renormalization of standard families} As before, fix $K_0>1$ and let $\mathcal{K}$ be the space of normalized $C^3$ critical commuting pairs which are $K_0$-controlled (see Definition \ref{defcont} in Section \ref{Seccont}). Again we consider in this section a normalized $C^3$ critical commuting pair $\zeta=(\eta,\xi)$ in $\mathcal{K}$ with negative Schwarzian, which is renormalizable with some period $a\in\nt$. Let $V(\zeta)$ be the parameter interval for the standard family around $\zeta$ constructed in Section \ref{subMF}, and consider the one-parameter family of $C^3$ critical commuting pairs given by  $G_t=p\mathcal{R}(\zeta_t)$ for each $t \in V$, that is, $G_t$ is the pre-renormalization of $\zeta_t$ (see Definition \ref{renop} in Section \ref{Sprel}).

\begin{prop}\label{geomprop} There exists $K=K(\mathcal{K})>1$ such that for all $t \in V$ and for all $x$ in the domain of $G_t$ we have:$$\frac{\partial G_t}{\partial t}(x) \asymp a^3\quad\mbox{if $x<0$, and}\quad\frac{\partial G_t}{\partial t}(x) \asymp 1\quad\mbox{if $x>0$.}$$
\end{prop}

\begin{figure}[h]
\begin{center}
\includegraphics[scale=1.3]{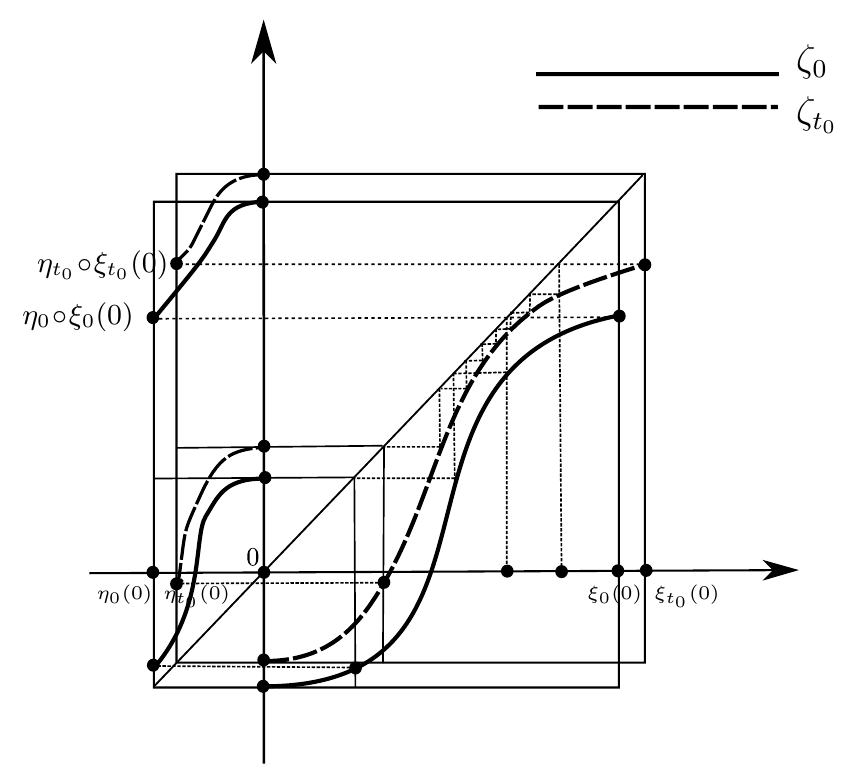}
\caption{Both critical commuting pairs of Figure 4, and their renormalizations.}
\end{center}
\end{figure}

\begin{proof}[Proof of Proposition \ref{geomprop}] Note first that for $t \in V$ and $x \in I_{\xi_t}$ we have the identity:
\begin{equation}\label{derGt}
\frac{\partial G_t}{\partial t}(x)=\frac{\partial\xi_t}{\partial t}(x)\,D\eta_t^{a}\big(\xi_t(x)\big)+\sum_{k=1}^{k=a}\frac{\partial T}{\partial t}\big(\eta_0\big(\eta_t^{k-1}\big(\xi_t(x)\big)\big),t\big)D\eta_t^{a-k}\big(\eta_t^{k}\big(\xi_t(x)\big)\big).
\end{equation}

Indeed, fix $x \in I_{\xi_t}$ and for each $j\in\{0,1,...,a\}$ let $y_j(t)=\eta_t^j\big(\xi_t(x)\big)$. Note that $y_0(t)=\xi_t(x)$ and $y_a(t)=G_t(x)$ for $x<0$. Since $y_{j+1}(t)=\eta_t\big(y_j(t)\big)=T\big(\eta_0\big(y_j(t)\big),t\big)$ for all $j\in\{0,1,...,a-1\}$ we see that:
\begin{align}\label{MGumpaso}
y'_{j+1}(t)&=y'_j(t)\frac{\partial T}{\partial x}\big(\eta_0\big(y_j(t)\big),t\big)D\eta_0\big(y_j(t)\big)+\frac{\partial T}{\partial t}\big(\eta_0\big(y_j(t)\big),t\big)\\
&=y'_j(t)D\eta_t\big(y_j(t)\big)+\frac{\partial T}{\partial t}\big(\eta_0\big(y_j(t)\big),t\big),\notag
\end{align}
since from $\eta_t(x)=T\big(\eta_0(x),t\big)$ we get $D\eta_t(x)=\frac{\partial T}{\partial x}\big(\eta_0(x),t\big)D\eta_0(x)$. By induction on \eqref{MGumpaso} we obtain that for all $j\in\{1,...,a\}$:
\begin{align*}
y'_j(t)&=y_0'(t)\prod_{l=0}^{j-1}D\eta_t\big(y_l(t)\big)+\sum_{k=1}^{j-1}\frac{\partial T}{\partial t}\big(\eta_0\big(y_{k-1}(t)\big),t\big)\prod_{l=k}^{j-1}D\eta_t\big(y_l(t)\big)\\
&+\frac{\partial T}{\partial t}\big(\eta_0\big(y_{j-1}(t)\big),t\big)\\
&=y_0'(t)\,D\eta_t^{j}\big(y_0(t)\big)+\sum_{k=1}^{j-1}\frac{\partial T}{\partial t}\big(\eta_0\big(y_{k-1}(t)\big),t\big)D\eta_t^{j-k}\big(y_k(t)\big)\\
&+\frac{\partial T}{\partial t}\big(\eta_0\big(y_{j-1}(t)\big),t\big)\\
&=y_0'(t)\,D\eta_t^{j}\big(y_0(t)\big)+\sum_{k=1}^{k=j}\frac{\partial T}{\partial t}\big(\eta_0\big(y_{k-1}(t)\big),t\big)D\eta_t^{j-k}\big(y_k(t)\big).
\end{align*}

In particular:$$\frac{\partial G_t}{\partial t}(x)=y'_a(t)=y_0'(t)\,D\eta_t^{a}\big(y_0(t)\big)+\sum_{k=1}^{k=a}\frac{\partial T}{\partial t}\big(\eta_0\big(y_{k-1}(t)\big),t\big)D\eta_t^{a-k}\big(y_k(t)\big),$$and then we obtain for all $t \in V$ and all $x \in I_{\xi_t}$ the desired identity \eqref{derGt}. Now by Lemma \ref{MFcontrolP}, the $K_0$-control and Lemma \ref{BDinV} we have:
\begin{align*}
0\leq\frac{\partial\xi_t}{\partial t}(x)D\eta_t^{a}\big(\xi_t(x)\big)&=\left(\frac{D\eta_0\big(\eta_t^{a-1}(\xi_t(x))\big)DP\big(\eta_0\big(\eta_t^{a-1}(\xi_t(x))\big)\big)}{DP\big(\xi_t(x)\big)DP\big(\eta_t^{a}\big(\xi_t(x)\big)\big)}\right)D\eta_t^{a-1}\big(\xi_t(x)\big)\\
&\leq KD\eta_0\big(\eta_t^{a-1}(\xi_t(x))\big)D\eta_t^{a-1}\big(\xi_t(x)\big)\leq K\frac{\big|I_a(t)\big|}{\big|I_1(t)\big|}\leq K.
\end{align*}

On the other hand, for all $k\in\{1,...,a\}$ we have:$$\frac{\partial T}{\partial t}\big(\eta_0\big(\eta_t^{k-1}\big(\xi_t(x)\big)\big),t\big)=\frac{1}{DP\big(\eta_t^{k}\big(\xi_t(x)\big)\big)}\in\left[\frac{1}{K},K\right]$$again by Lemma \ref{MFcontrolP}. Therefore, it follows from \eqref{derGt} that for any $x<0$ we have:$$\frac{\partial G_t}{\partial t}(x)\asymp\sum_{k=1}^{a-1}\,D\eta_t^{a-k}\big(\eta_t^k\big(\xi_t(x)\big)\big)\quad\mbox{whenever $a>1$.}$$

Again by Lemma \ref{BDinV} (bounded distortion) and the $K_0$-control we have that:$$\frac{\partial G_t}{\partial t}(x)\asymp\sum_{k=1}^{a-1}\frac{|I_a(t)|}{|I_k(t)|}\asymp\sum_{k=1}^{a-1}\frac{1}{|I_k(t)|}\,.$$

Therefore, by Yoccoz's lemma (Lemma \ref{doyoc}) we obtain:$$\frac{\partial G_t}{\partial t}(x)\asymp\sum_{k=1}^{a-1}\min\{k,a-k\}^2\asymp a^3\quad\mbox{for any $x<0$.}$$

Finally, recall that for $x\in\big[0,\eta_t^a\big(\xi_t(0)\big)\big]$ we have $G_t(x)=\eta_t(x)$ and then:$$\frac{\partial G_t}{\partial t}(x)=\frac{\partial\eta_t}{\partial t}(x)=\frac{1}{DP\big(\eta_t(x)\big)}\in\left[\frac{1}{K},K\right]$$by Lemma \ref{MFcontrolP}.
\end{proof}

With Proposition \ref{geomprop} at hand we obtain:

\begin{coro}\label{MFprimcor} There exists $K(\mathcal{K})>1$ such that for all $t \in V$ and $x,y \in I_{\xi_t}$ we have:$$\frac{\left|\frac{\partial G_t}{\partial t}(x)\right|}{\left|\frac{\partial G_t}{\partial t}(y)\right|} \leq K\,.$$

In particular:$$\frac{\big|G_t(x)-G_0(x)\big|}{\big|G_t(y)-G_0(y)\big|}=\frac{\big|\eta_t^a\big(\xi_t(x)\big)-\eta_0^a\big(\xi_0(x)\big)\big|}{\big|\eta_t^a\big(\xi_t(y)\big)-\eta_0^a\big(\xi_0(y)\big)\big|} \leq K$$for all $t \in V\!\setminus\!\{0\}$ and $x,y\in I_{\xi_t} \cap I_{\xi_0}=\big[\!\max\{\eta_0(0),\eta_t(0)\},0\big]$.
\end{coro}

\section{Orbit Deformations}\label{orbdef}

We start this section with the following observation:

\begin{lema}\label{op} Given $K_0>1$ there exist $a_0=a_0(K_0)\in\nt$ and $K=K(K_0)>1$ with the following property: let $\zeta=(\eta,\xi)$ be a normalized $C^3$ critical commuting pair with negative Schwarzian which is $K_0$-controlled and renormalizable with some period $a \geq a_0$. Then there exists a unique $p$ in $I_{\eta}$ such that $\big|\eta(p)-p\big|\leq\big|\eta(x)-x\big|$ for all $x \in I_{\eta}$. Moreover, the point $p$ belongs to the interior of $I_{\eta}$, $D\eta(p)=1$ and $D^2\eta(p)<-1/K<0$.
\end{lema}

\begin{proof}[Proof of Lemma \ref{op}] Since $\zeta$ is renormalizable we know that $x>\eta(x)$ for all $x \in I_{\eta}$. From the continuity of $\eta$ and the compactness of its domain $I_{\eta}$, we obtain the existence of a point $p$ such that $0<\big|\eta(p)-p\big|\leq\big|\eta(x)-x\big|$ for all $x \in I_{\eta}$.

We claim first that if $a_0>K_0^2$ and $a \geq a_0$, then $p$ belongs to the interior of $I_{\eta}$. Indeed, note first that the (positive) difference $\Id-\eta$ equals $|I_{\xi}|$ at the origin, and equals $\big|\xi(I_{\xi})\big|$ at the point $\xi(0)$. In both cases it is greater than $1/K_0$, by the $K_0$-control hypothesis. If $p$ is one of the boundary points of $I_{\eta}$, we would have $\big|\eta(x)-x\big|\geq 1/K_0$ for all $x \in I_{\eta}$, and since the period of $\zeta$ is $a$, we would have $a/K_0<|I_{\eta}|$. On the other hand, again by the $K_0$-control hypothesis, we have $a_0>K_0^2>K_0|I_{\eta}|$ and then $|I_{\eta}|<a_0/K_0$, which gives the desired contradiction.

With the claim at hand we clearly have $D\eta(p)=1$ and $D^2\eta(p) \leq 0$.

Uniqueness of $p$ follows at once from the Minimum Principle for maps with negative Schwarzian derivative (see \cite[Section II.6, Lemma 6.1]{dmvs} for its statement).

Now we claim that $D^2\eta(p)$ is strictly negative. Indeed, if $D^2\eta(p)=0$ we would have $D^3\eta(p)=S\eta(p)<0$, and then it would exist $\delta_0>0$ such that $D^2\eta(x)>0$ for all $x\in(p-\delta_0,p)$. But then it would exist $0<\delta_1\leq\delta_0$ such that $\big|\eta(x)-x\big|<\big|\eta(p)-p\big|$ for all $x\in(p-\delta_1,p)$, which gives the desired contradiction.

Finally, the fact that $D^2\eta(p)$ is uniformly bounded away from zero (by a constant depending only on $K_0$) follows from (the proof of) Yoccoz's lemma (Lemma \ref{doyoc}), see \cite[Appendix B, pages 386-389]{dfdm1}.
\end{proof}

Throughout this section fix $K_0>1$ and let $\mathcal{K}$ be the space of normalized $C^3$ critical commuting pairs which are $K_0$-controlled (see Definition \ref{defcont} in Section \ref{Seccont}). Let $\zeta=(\eta,\xi)$ and $\tilde{\zeta}=(\tilde{\eta},\tilde{\xi})$ be two $C^3$ critical commuting pairs with negative Schwarzian that belong to $\mathcal{K}$ which are renormalizable with the same period $a\in\nt$. Denote by $\varepsilon>0$ the $C^2$ distance between $\zeta$ and $\tilde{\zeta}$, that is, $\varepsilon=d_2(\zeta,\tilde{\zeta})$. We will assume that $\varepsilon<\varepsilon_0$, where $\varepsilon_0>0$ will be fixed later in this section. Moreover, we will only consider in this section the special situation when:
\begin{itemize}
\item[1)] $I_\eta=I_{\tilde{\eta}}$ and $I_\xi=I_{\tilde{\xi}}=[-1,0]$,
\item[2)] $p=\tilde{p}$ where $D\eta(p)=D\tilde{\eta}(\tilde{p})=1$ (see Lemma \ref{op} above).
\end{itemize}

Let $H:I_\eta\to[-\varepsilon,\varepsilon]\subset\R$ be defined by
$
H(x)=\eta(x)-\tilde{\eta}(x)
$
and let$$h=H(p).$$

Observe that for every $x\in I_{\eta}$ we have:
\begin{equation}\label{Hh}
\big|H(x)\big|\leq|h|+\varepsilon(x-p)^2,
\end{equation}
and
\begin{equation}\label{DH}
\big|DH(x)\big|\leq\varepsilon|x-p|.
\end{equation}

Indeed, given $x \in I_{\eta}$ there exists $y \in I_{\eta}$ such that $DH(x)=D^2H(y)(x-p)$ and then $\big|DH(x)\big|=\big|D^2H(y)\big||x-p|\leq\varepsilon|x-p|$, and there exists also $z\in[p,x]\subset I_{\eta}$ such that $H(x)=h+DH(z)(x-p)$ and then $\big|H(x)\big|\leq|h|+\big|DH(z)\big||x-p|\leq|h|+\varepsilon(x-p)^2$.

As before we will use the following notation: for $i\in\{0,...,a\}$ let $x_i=\eta^{i}(\xi(0))$. Note that $x_i\in I_{\eta}=[0,\xi(0)]$ for all $i\in\{0,...,a\}$. Define $\tilde{x}_i=\tilde{\eta}^{i}(\tilde{\xi}(0))$ similarly. Denote by $I_i$, $i\in\{1,...,a\}$, the fundamental domains of $\eta$ given by $I_i=[\eta^{i}(\xi(0)),\eta^{i-1}(\xi(0))]$. By the commuting condition $I_1=\xi(I_{\xi})=\xi ([-1,0])$. Define $\tilde{I}_i$ similarly. Let us state some consequences of Yoccoz's Lemma (Lemma \ref{doyoc}):

\begin{lema}\label{aquelevelho} There exists $K_0=K_0(\mathcal{K})>1$ such that for any $\zeta\in\mathcal{K}$ renormalizable with period $a\in\nt$, and for any $b < \left\lfloor\frac{a}{2}\right\rfloor$ we have:$$|x_{b}-x_{a-b}|\leq\frac{K_0}{b}\,.$$
\end{lema}

The constant $K_0$ does not depend on the period $a$.

\begin{proof}[Proof of Lemma \ref{aquelevelho}] By Yoccoz's Lemma (Lemma \ref{doyoc}) we have:$$|x_{b}-x_{a-b}|=\sum_{i=b+1}^{i=a-b}|I_i| \leq K\left(\sum_{i=b+1}^{i=a-b}\frac{1}{\min\{i,a-i\}^2}\right).$$

To finish, note that:$$\sum_{i=b+1}^{i=a-b}\frac{1}{\min\{i,a-i\}^2}\leq\frac{2}{b}\,.$$

Indeed, by symmetry it is enough to prove that:$$\sum_{i=b+1}^{i=\lfloor a/2 \rfloor}\frac{1}{\min\{i,a-i\}^2}\leq\frac{1}{b}\,,$$and this follows from elementary calculus:$$\sum_{i=b+1}^{i=\lfloor a/2 \rfloor}\frac{1}{\min\{i,a-i\}^2}=\sum_{i=b+1}^{i=\lfloor a/2 \rfloor}\frac{1}{i^2}\leq\int_{b}^{\lfloor a/2 \rfloor}\frac{dt}{t^2}\leq\int_{b}^{+\infty}\frac{dt}{t^2}=\frac{1}{b}\,.$$
\end{proof}

Another consequence of Yoccoz's lemma is the following:

\begin{lema}\label{elb} There exists $b=b(\mathcal{K})\in\nt$ such that $\tilde{x}_{\tilde{N}-b} \geq x_{N-1}$ and $\tilde{x}_{\tilde{N}+b} \leq x_{N+2}$.
\end{lema}

The distance between corresponding critical iterates of $\zeta$ and $\tilde{\zeta}$ will be denoted by $\Delta x_i$, that is:$$\Delta x_i=\tilde{x}_i-x_i=\tilde{\eta}^{i}\big(\tilde{\xi}(0)\big)-\eta^{i}\big(\xi(0)\big)\quad\mbox{for all $i\in\{0,1,...,a\}$.}$$

\begin{lem}\label{entrada} There exists $K=K(\mathcal{K})>0$ such that for $i \leq \min\big\{\lfloor a/2 \rfloor,N-b,\tilde{N}-b\big\}$ we have:$$|\Delta x_i|\le K\left(|h|\cdot i+\frac{\varepsilon}{i}\right).$$
\end{lem}

\begin{proof}[Proof of Lemma \ref{entrada}] Let $x_0=\xi(0)=\tilde{\xi}(0)$ be the common critical value of $\xi$ and $\tilde{\xi}$, which is the right boundary point of $I_\eta=I_{\tilde{\eta}}$. Recall that, by definition, $x_i=\eta^i(x_0)$ and $\tilde{x}_i=\tilde{\eta}^i(x_0)$ for all $i\in\big\{1,...,a\big\}$. We will consider the case $x_{\lfloor a/2 \rfloor} \leq \tilde{x}_{\lfloor a/2 \rfloor}$. Note that for any $i\in\big\{1,...,\lfloor a/2 \rfloor\big\}$ and any $k\in\{0,...,i-1\}$ we have by combinatorics:$$x_{a-i+k+1} \leq x_{\lfloor a/2 \rfloor+1} < x_{\lfloor a/2 \rfloor} \leq \tilde{x}_{\lfloor a/2 \rfloor} \leq \tilde{x}_{k+1}<\tilde{x}_{k}\,.$$

Therefore $x_{\lfloor a/2 \rfloor+1}<\eta(\tilde{x}_{k})$ and then $x_{a-i+k+1}<\eta(\tilde{x}_{k})$, that is, both points $\eta\big(\tilde{\eta}^k(x_0)\big)$ and $\tilde{\eta}^{k+1}(x_0)$ lie to the right of the point $x_{a-i+k+1}$. In particular the iterate $\eta^{i-k-1}$ is well defined in the interval with boundary points $\eta\big(\tilde{\eta}^k(x_0)\big)$ and $\tilde{\eta}^{k+1}(x_0)$. This allows us to use a simple telescopic trick and the mean-value theorem in order to write for any $i\in\big\{1,...,\lfloor a/2 \rfloor\big\}$:
\begin{align}\label{edwel}
|\Delta x_i|&=\left|\sum_{k=0}^{i-1}\big(\eta^{i-k-1}\big(\eta\big(\tilde{\eta}^{k}(x_0)\big)\big)-\eta^{i-k-1}\big(\tilde{\eta}^{k+1}(x_0)\big)\big)\right|\\
&\leq\sum_{k=0}^{i-1}\big|D\eta^{i-k-1}(y_k)\big|\big|H\big(\tilde{\eta}^{k}(x_0)\big)\big|,\notag
\end{align}
where for each $k\in\{0,...,i-1\}$ the point $y_k$ lies between $\eta\big(\tilde{\eta}^k(x_0)\big)$ and $\tilde{\eta}^{k+1}(x_0)$ (the points $y_0,y_1,...,y_{i-1}$ depends also on each fixed $i$, but we will denote them just by $y_k$ to simplify the notation). From \eqref{Hh} and Lemma \ref{aquele} we get that:
\begin{equation}\label{estemH}
\big|H\big(\tilde{\eta}^{k}(x)\big)\big| \leq |h|+\frac{K\varepsilon}{(k+1)^2}\,.
\end{equation}

For each $k\in\{0,...,i-1\}$ let us denote $D_k=\big|D\eta^{i-k-1}(y_k)\big|$. Our goal is, therefore, to estimate the sum:
\begin{equation}\label{lasuma}
\big|\eta^i(x)-\tilde{\eta}^i(x)\big|\leq\sum_{k=0}^{i-1}D_k\left(|h|+\frac{K\varepsilon}{(k+1)^2}\right).
\end{equation}

For each $k\in\{0,...,i-1\}$ let $m=m(k)\in\{1,...,a\}$ be such that $y_k \in I_m(\eta)$, where $I_m(\eta)=\big[\eta^{m}(x),\eta^{m-1}(x)\big]$ as before. Since we are assuming $x_{\lfloor a/2 \rfloor} \leq \tilde{x}_{\lfloor a/2 \rfloor}$ we have that $m \leq a/2+1$. We claim that $m(k) \asymp k$ for all $k\in\{0,...,i-1\}$, more precisely:

\begin{claim}\label{lacompdelmyk} There exists $C=C(\mathcal{K})>1$ such that $\frac{k}{C}<m<Ck$ for all $k\in\{0,...,i-1\}$ and for all $i\in\big\{1,...,\lfloor a/2 \rfloor\big\}$.
\end{claim}

\begin{proof}[Proof of Claim \ref{lacompdelmyk}] From Lemma \ref{aquele} we know that $|y_k-p|\asymp\frac{1}{m}$, and then it is enough to prove that $|y_k-p|\asymp\frac{1}{k}$. Recall that $d_2(\zeta,\tilde{\zeta})<\varepsilon_0$, where $\varepsilon_0>0$ will be fixed later in the proof. On one hand $|y_k-p|\leq\big|\tilde{\eta}^{k}(x_0)-p\big|\asymp\frac{1}{k}$. On the other hand, since $i \leq \min\big\{N-b,\tilde{N}-b\big\}$, the point $p$ does not belong to the interval with boundary points $\eta\big(\tilde{\eta}^k(x_0)\big)$ and $\tilde{\eta}^{k+1}(x_0)$, and then:
\begin{align*}
|y_k-p|&\geq\min\big\{\big|\tilde{\eta}^{k+1}(x_0)-p\big|,\big|\eta\big(\tilde{\eta}^k(x_0)\big)-p\big|\big\}\\
&=\big|\tilde{\eta}^{k+1}(x_0)-p\big|-\big|\eta\big(\tilde{\eta}^k(x_0)\big)-\tilde{\eta}^{k+1}(x_0)\big|\\
&=\big|\tilde{\eta}^{k+1}(x_0)-p\big|-\big|H\big(\tilde{\eta}^k(x_0)\big)\big|.
\end{align*}

From \eqref{estemH} we get $\big|H\big(\tilde{\eta}^{k}(x_0)\big)\big|\leq K\big(|h|+\frac{\varepsilon}{(k+1)^2}\big)\leq \frac{K}{(k+1)^2}$ since $|h| \leq K/a^2$ by Yoccoz's lemma (indeed, by Lemma \ref{doyoc}, the length of the fundamental domain $\big(\eta(p),p\big)$ is bounded by $1/a^2$, up to a multiplicative constant. That is, both $p-\eta(p)$ and $p-\tilde{\eta}(p)$ are bounded by $1/a^2$ up to a multiplicative constant, and then $|h| \leq K/a^2$). Therefore:
$$|y_k-p|\geq \frac{1}{K}\left(\frac{1}{k+1}-\frac{K^2}{(k+1)^2}\right)=\frac{1}{K}\left( 1-\frac{K^2}{k+1}\right)\frac{k}{k+1}\frac{1}{k}\geq \frac{1}{4k}$$ if 
$k\geq 2K^2+1$ and then $|y_k-p|\asymp\frac{1}{k}$ in this case. We choose $\varepsilon_0>0$ in order to have that if $k\leq 2K^2+1$  then both $\tilde{\eta}^{k+1}(x)$ and 
$\eta (\tilde{\eta}^{k}(x))$ belong to the interval $\big[\tilde{\eta}^{k+2}(x), \tilde{\eta}^{k}(x)\big]$ and  again $|y_k-p|\asymp\frac{1}{k}$ as we wanted to prove.
\end{proof}

We have two claims regarding the values of $D_k$:

\begin{claim}\label{limsempre} There exists $K=K(\mathcal{K})>0$ such that for all $k\in\{0,...,i-1\}$ and $i\in\{1,...,a/2\}$ we have $D_k \leq K$.
\end{claim}

\begin{proof}[Proof of Claim \ref{limsempre}] By bounded distortion and Yoccoz's lemma we know that:$$\big|D\eta^{i-k-1}(y_k)\big|\asymp\frac{\big|I_{m+i-k-1}(\eta)\big|}{\big|I_{m}(\eta)\big|}\asymp m^2\big|I_{m+i-k-1}(\eta)\big|,$$and then it is enough to prove that $\big|I_{m+i-k-1}(\eta)\big|\leq\frac{K}{m^2}$. To prove this we have two cases to consider: if $\eta^{i-k-1}(y_k) \geq p$ then $\big|I_{m+i-k-1}(\eta)\big|\asymp\frac{1}{(m+i-k-1)^2}$ by Yoccoz's lemma, and since $i-k-1 \geq 0$ we are done. If $\eta^{i-k-1}(y_k)<p$ then $\big|I_{m+i-k-1}(\eta)\big|\asymp\frac{1}{(a-m-i+k+1)^2}$, and since $a-m-i \geq 0$ we obtain $\big|I_{m+i-k-1}(\eta)\big|\leq\frac{K}{(k+1)^2}$. Since $m \asymp k$ by Claim \ref{lacompdelmyk}, we obtain Claim \ref{limsempre}.
\end{proof}

\begin{claim}\label{afwel} There exists $K=K(\mathcal{K})>0$ such that if $k<\frac{i}{4(C-1)}$ then $D_k \leq K\frac{k^2}{i^2}$.
\end{claim}

\begin{proof}[Proof of Claim \ref{afwel}] Write $m=\lfloor\theta k\rfloor$ with $\frac{1}{C}<\theta<C$ (see Claim \ref{lacompdelmyk}). If $m<k$ we have that $\theta<1$ and $i+m-k-1= \theta i + (1-\theta)i- (1-\theta)k-1=\theta i +(1-\theta)(i-k)-1\geq \theta i -1\geq \frac {1}{C} i$. Since $i+m-k-1\leq i \leq \frac {a}{2}$ we have that $D_k \leq K\frac{C^2k^2}{(\frac{i}{C})^2} \leq K\frac {k^2}{i^2}$. On the other hand, if $m>k$ (that is, $\theta >1$), we have $m+i-k-1 \leq i+ (\theta-1)k-1 \leq i+(C-1)k-1 \leq i  + \frac {1}{4} i-1 \leq \frac {3}{2} a$. Then $\big|I_{m+i-k-1}(\eta)\big|\asymp\frac{1}{(m+i-k-1)^2}\leq\frac{1}{(i-1)^2}$, and so we also have $D_k \leq K\frac{k^2}{i^2}$ in this case, since $\frac {1}{3} a< j<\frac {2}{3} a$ implies $\frac{1}{a-j}> \frac {1}{2} \frac{1}{j}> \frac {1}{4} \frac{1}{a-j}$.
\end{proof}

With Claim \ref{limsempre} and Claim \ref{afwel} at hand we are ready to estimate the sum \eqref{lasuma}:
\begin{align*}
\sum_{k=0}^{i-1}D_k\left(|h|+\frac{K\varepsilon}{(k+1)^2}\right)&=|h|\left(\sum_{k=0}^{i-1}D_k\right)+K\varepsilon\left(\sum_{k=0}^{\big\lfloor\frac{i}{4(C-1)}\big\rfloor}\frac{D_k}{(k+1)^2}\right)\\
&+K\varepsilon\left(\sum_{k=\big\lfloor\frac{i}{4(C-1)}\big\rfloor+1}^{i-1}\frac{D_k}{(k+1)^2}\right)\\
&\leq K|h|i+K\frac{\varepsilon}{i^2}\left(\sum_{k=0}^{\big\lfloor\frac{i}{4(C-1)}\big\rfloor}\left(\frac{k}{k+1}\right)^2\right)\\
&+K\varepsilon\left(\sum_{k=\big\lfloor\frac{i}{4(C-1)}\big\rfloor+1}^{i-1}\frac{1}{(k+1)^2}\right)\\
&\leq K|h|i+K\frac{\varepsilon}{i}+K\frac{\varepsilon}{i}\,.
\end{align*}

For the last inequality we have used that both sequences$$\frac{1}{i}\left(\sum_{k=0}^{\big\lfloor\frac{i}{4(C-1)}\big\rfloor}\left(\frac{k}{k+1}\right)^2\right) \quad\mbox{and}\quad i\left(\sum_{k=\big\lfloor\frac{i}{4(C-1)}\big\rfloor+1}^{i-1}\frac{1}{(k+1)^2}\right)$$remain bounded when $i$ goes to infinity, with constants depending only on $C$. We have proved Lemma \ref{entrada}.
\end{proof}

\begin{lem}\label{bruteforce} For every $a\ge 1$ there exists $K_a>0$ such that
$$
|\Delta x_a|\le K_a \varepsilon.
$$
\end{lem}

\begin{proof}[Proof of Lemma \ref{bruteforce}] Observe,
$$
\begin{aligned}
|\Delta_{i+1}|&=|\tilde{\eta}(\tilde{x}_i)-\eta(x_i)|\\
&=|\eta(\tilde{x}_i)-\eta(x_i)+H(\tilde{x}_i)|\\
&\le D |\Delta x_i| +\varepsilon,
\end{aligned}
$$
where $D=\max \{D\eta\}$. So
$$
|\Delta x_a|\le \varepsilon \cdot \sum_{k=0}^{a} D^{a-k}.
$$
The Lemma follows.
\end{proof}

The following definition is given for general commuting pairs which are contained in the previously discussed set $\mathcal{K}$ of $K_0$-controlled commuting pairs.

\begin{defn}\label{defsync} Given $L>1$ we say that the commuting pairs $\zeta_0=(\xi_0,\eta_0)$ and $\zeta_1=(\xi_1,\eta_1)$, with $a_{\zeta_0}=a_{\zeta_1}=a$, 
are $L$-synchronized if$$|\Delta x_a|\le  L\cdot  d_2(\zeta_0,\zeta_1).$$ 
\end{defn}

By working just as in the proof of Lemma \ref{entrada} but with backwards iterations we obtain:

\begin{lema}\label{backward} Given $L>0$ there exists $K=K(\mathcal{K},L)>0$ such that if $\zeta,\tilde{\zeta}\in\mathcal{K}$ are $L$-synchronized with $a_{\zeta}=a_{\tilde{\zeta}}=a$, then we have:$$|\Delta x_i|\le K\left(|h|(a-i)+\frac{\varepsilon}{a-i}\right)$$for all $i\in\nt$ such that $\max\big\{\lfloor a/2 \rfloor,N+b,\tilde{N}+b\big\}\leq i \leq a$.
\end{lema}

\begin{prop}\label{h} For every $L>0$ there exists $K=K(\mathcal{K},L)>1$ such that the following holds. If $\zeta,\tilde{\zeta}\in\mathcal{K}$ are $L$-synchronized with $a_{\zeta}=a_{\tilde{\zeta}}=a$, then we have:$$|h|\le K\frac{\varepsilon}{a^2}\,.$$
\end{prop}

\begin{proof}[Proof of Proposition \ref{h}] Let us suppose that $\tilde{\eta}(p)=\eta(p)+h$ with $h>0$. We want to prove that, under the synchronization assumption, the ratio $C=\frac{a^2h}{\varepsilon}$ is uniformly bounded in $\mathcal{K}$.

Let $N\in\{1,...,a\}$ defined by $p\in[x_{N+1},x_{N}]$. By Yoccoz's lemma (see in particular \cite[Lemma B.1, page 387]{dfdm1}) there exists $K_0=K_0(\mathcal{K})>1$ such that $N=\nu a$ with $1/K_0 \leq \nu \leq 1-\frac{1}{K_0}$. In the same way let $\tilde{N}=\tilde{\nu}a$ defined by $p\in[\tilde{x}_{\tilde{N}+1},\tilde{x}_{\tilde{N}}]$ with $1/K_0 \leq \tilde{\nu} \leq 1-\frac{1}{K_0}$.

By Lemma \ref{aquelevelho} there exists $K_1=K_1(\mathcal{K})>1$ such that $(x_j,\tilde{x}_j)\subset\big(p-K_1/M,p\big)$ when $(1-\frac{1}{K_0})a \leq j \leq a-M$, and $(x_j,\tilde{x}_j)\subset\big(p,p+K_1/M\big)$ when $M \leq j \leq a/K_0$ for any $M\in\big\{1,...,\lfloor a/K_0 \rfloor\big\}$.

Let $K_2=K_2(\mathcal{K})>1$ be the constant given by Lemma \ref{entrada}. By Lemma \ref{backward} we have:
\begin{equation}\label{no414}
|\Delta x_{a-M}|\leq K_3\left(hM+\frac{\varepsilon}{M}\right)
\end{equation}
for some universal constant $K_3(L,\mathcal{K})>1$. Let $K=\max\{K_0,K_1,K_2,K_3\}$ and let us suppose that $a>K(4K+1)$ (otherwise we are done since $|h|\leq\varepsilon$). Fix $M\in\big\{1,...,\lfloor a/2 \rfloor\big\}$ small enough in order to have:$$0<\theta=\frac{M}{a}<\frac{1}{K(4K+1)}<1\,.$$

Let $T=\big[p-K/M,p+K/M\big]$ and recall that $(x_j,\tilde{x}_j) \subset T$ for all $j\in\{M,...,a-M\}$.

The next three claims will show that if $C$ is big enough, in terms of $K$ and $\theta(K)$, the pairs $\zeta$ and $\tilde{\zeta}$ cannot be $L$-synchronized.

\begin{claim}\label{4.10.1} If $C \geq 2\left(\frac{K}{\theta}\right)^2$, then $\tilde{\eta}(x)\geq\eta(x)+\frac{h}{2}$ for all $x \in T$.
\end{claim}

\begin{proof}[Proof of Claim \ref{4.10.1}] As before:$$\tilde{\eta}(x)-\eta(x)\geq h-\varepsilon(x-p)^2 \geq h-\varepsilon\left(\frac{K}{M}\right)^2=h-\frac{\varepsilon}{a^2}\left(\frac{K}{\theta}\right)^2\geq h-\frac{h}{2}=\frac{h}{2}\,.$$

In the last inequality we have used that $\frac{\varepsilon}{a^2}\leq\frac{h}{2}\left(\frac{\theta}{K}\right)^2$ since $\frac{a^2h}{\varepsilon} \geq 2\left(\frac{K}{\theta}\right)^2$.
\end{proof}

Note that $0<\theta<\frac{1}{K(4K+1)}$ implies $1-2\theta K^2-\theta K \in (0,1)$.

\begin{claim}\label{4.10.2} If $C>\frac{1}{\theta}\left(\frac{2K^2}{1-2\theta K^2-\theta K}\right)$ there exists $i_0\in\{M,...,a/K\}$ such that $x_{i_0}\leq\tilde{x}_{i_0}$.
\end{claim}

\begin{proof}[Proof of Claim \ref{4.10.2}] We will prove first that:
\begin{equation}\label{noclaim4.10.2}
\left(\frac{a}{K}-M\right)\frac{h}{2} \geq K\left(hM+\frac{\varepsilon}{M}\right).
\end{equation}

Indeed, since $1-2\theta K^2-\theta K>\frac{2K^2}{C\theta}$ we have:$$\frac{1-2\theta K^2-\theta K}{2\theta K}>\frac{K}{C\theta^2}$$and then:$$hM\left(\frac{1-2\theta K^2-\theta K}{2\theta K}\right)>K\frac{\varepsilon}{M}$$since $\varepsilon/M=hM/C\theta^2$. From:$$\frac{1-2\theta K^2-\theta K}{2\theta K}=\frac{1}{2}\left(\frac{1}{\theta K}-1-2K\right)$$we obtain:$$\frac{h}{2}\left(\frac{a}{K}-M\right)-KhM>K\frac{\varepsilon}{M}$$which implies the desired estimate \eqref{noclaim4.10.2}. Now estimate \eqref{noclaim4.10.2} combined with Lemma \ref{entrada} gives us:
\begin{equation}\label{no413}
\big|x_M-\tilde{x}_M\big|\leq\left(\frac{a}{K}-M\right)\frac{h}{2}\,.
\end{equation}

With estimate \eqref{no413} at hand we are ready to prove Claim \ref{4.10.2}. Indeed, let $i\in\{M,...,a/K\}$ be such that $p\leq\tilde{x}_i<x_i \leq p+K/M$ (if no such $i$ exists we are done). From Claim \ref{4.10.1} we have:
\begin{align*}
\tilde{x}_{i+1}-x_{i+1}&=\tilde{\eta}(\tilde{x}_i)-\eta(x_i)\geq h/2+\eta(\tilde{x}_i)-\eta(x_i)=h/2+D\eta(y_i)(\tilde{x}_{i}-x_{i})\\
&=h/2+\tilde{x}_{i}-x_{i}+D^2\eta(z_i)(y_i-p)(\tilde{x}_{i}-x_{i}),
\end{align*}
where $y_i\in[\tilde{x}_{i},x_{i}]$ and $z_i\in[p,y_i]$ are given by the mean-value theorem. Since $D^2\eta(z_i)<0$, $y_i-p>0$ and $\tilde{x}_{i}-x_{i}<0$ we obtain:$$\tilde{x}_{i+1}-x_{i+1} \geq h/2+\tilde{x}_{i}-x_{i}\,,\quad\mbox{that is:}\quad\Delta x_{i+1}\geq h/2+\Delta x_i\,.$$

Therefore if the difference $\tilde{x}_{i+1}-x_{i+1}$ is still negative, it will be at least $h/2$ closer to zero than the previous difference $\tilde{x}_{i}-x_{i}$. What estimate \eqref{no413} tells us is that we have enough time inside the interval $(p,p+K/M)$ in order to interchange the positions of the critical iterates. With this we have proved Claim \ref{4.10.2}.
\end{proof}

Claim \ref{4.10.2} implies that $x_{i}\leq\tilde{x}_{i}$ for all $i\in\{i_0,...,a-M\}$, since $D\eta>0$ and $h>0$. Therefore, by Claim \ref{4.10.1} we have:
\begin{equation}\label{ultdel410}
|\Delta x_{a-M}|\geq\frac{h}{2}\left[a-M-\left(1-\frac{1}{K}\right)a\right].
\end{equation}

Our third and last claim tells us that \eqref{ultdel410} contradicts the synchronization assumption. Note that $0<\theta<\frac{1}{K(4K+1)}$ implies $1-\theta K(4K+1)\in(0,1)$.

\begin{claim}\label{4.10.3}$$\mbox{If}\quad C \geq \frac{1}{\theta}\left[\frac{4K^2}{1-\theta K(4K+1)}\right],\quad\mbox{then}\quad 2K\left(hM+\frac{\varepsilon}{M}\right)\leq\frac{h}{2}\left[a-M-\left(1-\frac{1}{K}\right)a\right].$$
\end{claim}

\begin{proof}[Proof of Claim \ref{4.10.3}] Note first that:$$2K\left(hM+\frac{\varepsilon}{M}\right)=\frac{\varepsilon}{a}\left[2K\left(C\theta+\frac{1}{\theta}\right)\right]$$and$$\frac{h}{2}\left[a-M-\left(1-\frac{1}{K}\right)a\right]=\frac{\varepsilon}{a}\left[\frac{C}{2}\left(\frac{1}{K}-\theta\right)\right].$$

A straightforward computation shows that both conditions:$$C\geq\frac{1}{\theta}\left[\frac{4K^2}{1-\theta K(4K+1)}\right]\quad\mbox{and}\quad 2K\left(C\theta+\frac{1}{\theta}\right)\leq\frac{C}{2}\left(\frac{1}{K}-\theta\right)$$are actually equivalent.
\end{proof}

We are ready to finish the proof of Proposition \ref{h}. Indeed, by combining estimates \eqref{no414} and \eqref{ultdel410} we have:$$\frac{h}{2}\left[a-M-\left(1-\frac{1}{K}\right)a\right]\leq\big|x_{a-M}-\tilde{x}_{a-M}\big|\leq K\left(hM+\frac{\varepsilon}{M}\right)<2K\left(hM+\frac{\varepsilon}{M}\right)$$which contradicts Claim \ref{4.10.3}. Therefore:$$C\leq\max\left\{2\left(\frac{K}{\theta}\right)^2,\,\frac{1}{\theta}\left(\frac{2K^2}{1-2\theta K^2-\theta K}\right),\,\frac{1}{\theta}\left[\frac{4K^2}{1-\theta K(4K+1)}\right]\right\},$$that is, the ratio $C=\frac{a^2h}{\varepsilon}$ is bounded by a constant only depending on $K$ and $L$. We have proved Proposition \ref{h}.
\end{proof}

With Proposition \ref{h} at hand we can improve both Lemma \ref{entrada} and \ref{backward} under the synchronization assumption:

\begin{lema}\label{acl} Given $L>0$ there exists $K=K(\mathcal{K},L)>0$ such that if $\zeta,\tilde{\zeta}\in\mathcal{K}$ are $L$-synchronized with $a_{\zeta}=a_{\tilde{\zeta}}=a$, then we have:$$|\Delta x_i|\leq\frac{K\varepsilon}{i}\quad\mbox{for all $1\leq i \leq\min\big\{\lfloor a/2 \rfloor,N-b,\tilde{N}-b\big\}$, and}$$$$|\Delta x_i|\leq\frac{K\varepsilon}{a-i}\quad\mbox{for all $a\geq i \geq\max\big\{\lfloor a/2 \rfloor,N+b,\tilde{N}+b\big\}$.}$$
\end{lema}

Moreover:

\begin{prop}\label{Dxi} For every $L>0$ there exists $K=K(\mathcal{K}, L)>0$ such that the following holds. If $\zeta$ and $\tilde{\zeta}$ are $L$-synchronized then
$$
|\Delta x_i|\le  K \varepsilon \cdot \frac{1}{i}\quad\mbox{for all $i\in\{0,1,...,a/2\}$.}
$$
and
$$
|\Delta x_i|\le  K \varepsilon \cdot  \frac{1}{a-i}\quad\mbox{for all $i\in\{a/2,...,a\}$.}
$$
\end{prop}

\begin{proof}[Proof of Proposition \ref{Dxi}] By Lemma \ref{acl} we only need to estimate $|\Delta x_i|$ for the intermediate iterates $\min\big\{\lfloor a/2 \rfloor,N-b,\tilde{N}-b\big\}<i<\big\{\lfloor a/2 \rfloor,N+b,\tilde{N}+b\big\}$. We will prove only the first part of the statement (the other being the same), that is, we will prove that:$$|\Delta x_i|\leq  K\varepsilon\cdot\frac{1}{i}\quad\mbox{for all $i\in\big\{\min\{\lfloor a/2 \rfloor,N-b,\tilde{N}-b\},...,a/2\big\}$.}$$

We use the same notation as in the proof of Proposition \ref{h}. By the choice of $\theta$ we know that $M\leq\min\big\{\lfloor a/2 \rfloor,N-b,\tilde{N}-b\big\}$ and $a-M\geq\max\big\{\lfloor a/2 \rfloor,N+b,\tilde{N}+b\big\}$.

Recall that $H:I_\eta\to[-\varepsilon,\varepsilon]\subset\R$ is defined as $H(x)=\eta(x)-\tilde{\eta}(x)$. By Proposition \ref{h} we have that $\big|H(x)\big|\leq\varepsilon\big[\frac{K}{a^2}+(x-p)^2\big]$ and then $\big|H(x)\big|\leq\frac{K\varepsilon}{a^2}$ whenever $x \in T$, since for $x \in T$ we have that $|x-p|\leq\frac{K}{M}\leq\frac{K}{a}$. Therefore, by consider $\alpha=1+\frac{K}{a}$ and $\beta=\frac{K\varepsilon}{a^2}$, we obtain that $\Delta x_{i+1}\leq\alpha\Delta x_i+\beta$ and then:$$\Delta x_{i+n}\leq\alpha^n\Delta x_i+\beta\left(\sum_{j=0}^{n-1}\alpha^j\right)\quad\mbox{for all $1 \leq n \leq (\delta_1-\delta_0)a+2b$.}$$

Note that $\sum_{j=0}^{n-1}\alpha^j=\frac{\alpha^n-1}{\alpha-1}=\frac{a}{K}(\alpha^n-1)$. Moreover, since $n<a$ we have that $\alpha^n=(\frac{K}{a}+1)^n \leq e^{\frac{Kn}{a}}$ is bounded. Therefore:$$\Delta x_{i+n}\leq\alpha^n\Delta x_i+\beta\frac{a}{K}(\alpha^n-1) \leq K\frac{\varepsilon}{i}\left[\alpha^n+\frac{i}{a}(\alpha^n-1)\right]\leq K\frac{\varepsilon}{i}\alpha^n\leq K\frac{\varepsilon}{i}\,.$$

Finally, from $i \geq M = \theta a$ and $n \leq (\delta_1-\delta_0)a+2b$ we get that $\frac{n}{i}$ is bounded and then $\Delta x_{i+n} \leq K\frac{\varepsilon}{i+n}$ as we wanted to prove.
\end{proof}

For $i\in\{1,...,a\}$ let
$$
\Delta_i=|\Delta x_i-\Delta x_{i-1}|.
$$

\begin{prop}\label{Di}  For every $L>0$ there exists $K=K(\mathcal{K}, L)>0$ such that the following holds. If $\zeta$ and $\tilde{\zeta}$ are $L$-synchronized then
$$
\Delta_i\le  K\left( \varepsilon \cdot \frac{ \log i}{i^2} +  \varepsilon^2
\cdot \frac{1}{i}\right)\quad\mbox{for all $i\le a/2$.}
$$
and
$$
\Delta_i\le  K\left( \varepsilon \cdot \frac{ \log (a-i)}{(a-i)^2} +
\varepsilon^2 \cdot \frac{1}{a-i}\right)\quad\mbox{for all $i\ge a/2$.}
$$
\end{prop}

\begin{proof}[Proof of Proposition \ref{Di}] The proof of the second part of this proposition can be obtained as the first part by working backward. See also the proof of Proposition \ref{Dxi}. We will only present the proof of the first part. Observe, for $i\ge 1$,
$$
\begin{aligned}
\Delta_{i+1}=&|[\tilde{\eta}(x_i+\Delta x_i)-\eta(x_i)]-
[\tilde{\eta}(x_{i-1}+\Delta x_{i-1})-\eta(x_{i-1})]|\\
=&|[\eta(x_i+\Delta x_i)-\eta(x_i) +\tilde{\eta}(x_i+\Delta
   x_i)-\eta(x_i+\Delta x_i)]-\\
&[\eta(x_{i-1}+\Delta x_{i-1})-\eta(x_{i-1})
+\tilde{\eta}(x_{i-1}+\Delta x_{i-1})-\eta(x_{i-1}+\Delta x_{i-1})]|\\
=&|[D\eta(\theta_i)\Delta x_i +H(x_i+\Delta x_i)]-\\
&[D\eta(\theta_{i-1})\Delta x_{i-1} +H(x_{i-1}+\Delta x_{i-1})]|\\
\le& |D\eta(\theta_i)\Delta x_i-D\eta(\theta_{i-1})\Delta x_{i-1}|+|DH(\theta)\tilde{I}_i| \\
\end{aligned}
$$
The intermediate point $\theta$ is in $\tilde{I}_i$. Hence, by using
(\ref{DH}), the Yoccoz Lemma \ref{doyoc}, and Lemma \ref{aquele} we have
\begin{equation}\label{DHest}
|DH(\theta)\tilde{I}_i|\le K \varepsilon \cdot \frac{1}{i^3}.
\end{equation}
The intermediate point $\theta_i$ is in 
$[x_i,x_i+\Delta x_i]$. Similarly, $\theta_{i-1}\in [x_{i-1}, x_{i-1}+\Delta x_{i-1}]$. This allows
for the following estimate.
$$
\begin{aligned}
|D\eta(\theta_i)\Delta x_i-D\eta(\theta_{i-1})\Delta x_{i-1}|\le 
&\frac{|I_{i+1}|}{|I_i|}\Delta_i+ |(D\eta(\theta_i)-\frac{|I_{i+1}|}{|I_i|}    )\Delta x_i |+\\
&|(D\eta(\theta_{i-1})-\frac{|I_{i+1}|}{|I_i|} )\Delta x_{i-1}|\\
\le & \frac{|I_{i+1}|}{|I_i|}\Delta_i+K(|I_i|+ |\Delta x_i |) |\Delta
x_i |+\\
&K(|I_i|+ |\Delta x_{i-1} |) |\Delta x_{i-1} |\\
\end{aligned}
$$
Use the Yoccoz Lemma \ref{doyoc} and Proposition \ref{Dxi} to obtain
\begin{equation}\label{Ddetaest}
|D\eta(\theta_i)\Delta x_i-D\eta(\theta_{i-1})\Delta x_{i-1}|\le
\frac{|I_{i+1}|}{|I_i|}\Delta_i+
K( \varepsilon \frac{1}{i^3}+\varepsilon^2 \frac{1}{i^2}).
\end{equation}
Combine (\ref{DHest}) and (\ref{Ddetaest}) to obtain
$$
\Delta_{i+1}\le \frac{|I_{i+1}|}{|I_i|}\Delta_i+K( \varepsilon \frac{1}{i^3}+\varepsilon^2 \frac{1}{i^2}).
$$
After iterating this recursive estimate and using the Yoccoz Lemma
\ref{doyoc} one gets
$$
\begin{aligned}
\Delta_i&\le K\sum_{k=1}^{i-1}  (\varepsilon \frac{1}{k^3}+\varepsilon^2
\frac{1}{k^2})\cdot \frac{|I_i|}{|I_{k+1}|}\\
&\le K  (\varepsilon \frac{1}{i^2} \sum_{k=1}^{i-1} \frac{1}{k}+\varepsilon^2
\frac{1}{i^4} \sum_{k=1}^{i-1} k^2)\\
&\le K( \varepsilon \cdot \frac{ \log i}{i^2}+\frac{\varepsilon^2}{i}).
\end{aligned}
$$
\end{proof}

\section{Composition}\label{comp}

In this section we will discuss composition of multiple diffeomorphisms.
Let $I=[a,b]$ be a compact interval in the real line, and let $\mathcal{D}=\text{Diff}_+^2\big([a,b]\big)$ be the space of orientation preserving $C^2$ diffeomorphisms of $I$, endowed with the $C^2$-metric. Let $X=C^0\big(I,\R\big)$ be the vector space of continuous functions from $[a,b]$ to the real line, and recall that $X$ is a Banach space when endowed with the sup norm. Consider the \emph{non-linearity} $N:\mathcal{D} \to X$ defined as:$$N\psi=\frac{D^2\psi}{D\psi}=D\log D\psi\,.$$

Note that $N$ is a homeomorphism, whose inverse is given by:$$\big(N^{-1}\phi\big)(x)=a+\left(\frac{b-a}{\int_a^b\exp\big(\int_a^s\phi(t)dt\big)ds}\right)\int_a^x\exp\left(\int_a^s\phi(t)dt\right)ds\,,$$for any $x\in[a,b]$ and any $\phi \in X$. To prove that $N^{-1}\phi\in\mathcal{D}$ note that $DN^{-1}\phi>0$, since $\frac{\partial}{\partial x}\big(\int_a^x\exp\big(\int_a^s\phi(t)dt\big)ds\big)=\exp\big(\int_a^x\phi(t)dt\big)>0$.

In general, if $f:I\to\R$ is a $C^2$ map and $x$ is a regular point of $f$, we define $Nf(x)=D^2f(x)/Df(x)$. The chain rule for the non-linearity is $N(f \circ g)=Nf\!\circ\! g\,Dg+Ng$. The kernel of $N$ is the group of affine transformations. In particular $N(A\circ f)=Nf$ whenever $A$ is affine. Note also that the non-linearity goes to infinity around any non-flat critical point. Elementary properties of non-linearity can be found in \cite{marco}. On bounded sets it is bi-Lipschitz. In particular,

\begin{lem}\label{N-1} Let $B$ be a bounded set in $X=C^0\big(I,\R\big)$. There exists $K=K(B)>0$ such that for any pair $\phi$, $\psi$ in $B$ we have:$$d_2(N^{-1}\phi,N^{-1}\psi) \leq Kd_{C^0}(\phi,\psi).$$
\end{lem}

\begin{proof}[Proof of Lemma \ref{N-1}] Use the inverse of the non-linearity to estimate the $C^0$ distance between $f=N^{-1}\phi$ and $g=N^{-1}\psi$, as in \cite[Lemma 10.2, page 579]{marco}. This gives $d_{C^0}(N^{-1}\phi,N^{-1}\psi) \leq Kd_{C^0}(\phi,\psi)$. Since both $f=N^{-1}\phi$ and $g=N^{-1}\psi$ belong to $\Diff_+^2(I)$ there exists $t_0\in I$ such that $Df(t_0)=Dg(t_0)$, and then $\log Df(t)-\log Dg(t)=\int_{t_0}^{t}\big(\phi-\psi\big)(s)\,ds$ for all $t \in I$. Therefore $d_{C^0}(\log Df,\log Dg)\leq|I|d_{C^0}(\phi,\psi)$, and since both $f$ and $g$ are $C^1$-bounded we get $d_{C^0}(Df,Dg) \leq Kd_{C^0}(\phi,\psi)$. Finally note that for all $t \in I$ we have:$$\big|(D^2f-D^2g)(t)\big|\leq\big|(\phi-\psi)(t)\big|\big|Df(t)\big|+\big|(Df-Dg)(t)\big|\big|\psi(t)\big|.$$
\end{proof}

As we said before, the non-linearity allows us to identify the set $\mathcal{D}$ of diffeomorphisms with the Banach space $X=C^0\big(I,\R\big)$ of continuous functions. This defines the {\it non-linearity} norm on $\mathcal{D}$: $|f|=|Nf|_{C^0}$.

The following Lemma says that composition of multiple diffeomorphisms on $C^1$-bounded sets is Lipschitz continuous in the non-linearity norm. This Lemma is an adaptation of the Sandwich-Lemma in \cite[Lemma 10.5, page 581]{marco}.

\begin{lem}\label{NLcomp} Given $M>0$ there exist $K(M)>0$ such that for $f_1,...,f_n,g_1,...,g_n$ in $\Diff_+^3\big([0,1]\big)$ satisfying:
\begin{itemize}
\item $\sum_{j=1}^{j=n}|Nf_j|_{C^0} \leq M$,
\item $\sum_{j=1}^{j=n}|Ng_j|_{C^0} \leq M$,
\item $\sum_{j=1}^{j=n}|DNf_j|_{C^0} \leq M$,
\item $\sum_{j=1}^{j=n}|DNg_j|_{C^0} \leq M$
\end{itemize}
we have:$$\big| 
N\big(\bigcirc_{j=1}^{j=n}f_j\big)-N\big(\bigcirc_{j=1}^{j=n}g_j\big)\big|_{C^0} \leq K\sum_{j=1}^{j=n}\big|Nf_j-Ng_j\big|_{C^0}.$$
In particular,
$$
d_{C^2}\big(\bigcirc_{j=1}^{j=n}f_j,\bigcirc_{j=1}^{j=n}g_j\big) \leq K\sum_{j=1}^{j=n} \big|Nf_j-Ng_j\big|_{C^0}.$$
\end{lem}

The branches of renormalizations are compositions of a homeomorphism and multiple diffeomorphisms. The composition of multiple diffeomorphisms can be controlled by Lemma \ref{NLcomp}. To control the effect of the first factor we need the following Lemma, a basic property of composition.

\begin{lem}\label{comp2} For every $L>0$ there exists $K>0$ such that the following holds.
Let $q, \tilde{q}:[-1,0]\to [0,1]$ be $C^3$ homeomorphisms with one critical point, $Dq(0)=D\tilde{q}(0)=0$. Let $f, \tilde{f}:[0,1]\to [0,1]$ be $C^3$
diffeomorphisms. If $|q|_{C^3}$, $|\tilde{q}|_{C^3}$, $|f|_{C^3}$, $|\tilde{f}|_{C^3}\le L$ then
$$
d_{C^2}(\tilde{f}\circ \tilde{q}, f\circ q)\le K  
d_{C^2}(\tilde{f}, f)+d_{C^2}( \tilde{q}, q).
$$
\end{lem}

As before, fix $K_0>1$ and let $\mathcal{K}$ be the space of normalized $C^3$ critical commuting pairs which are $K_0$-controlled. Let $\zeta=(\eta,\xi)$ and $\tilde{\zeta}=(\tilde{\eta},\tilde{\xi})$ be two $C^3$ critical commuting pairs with negative Schwarzian that belong to $\mathcal{K}$ which are renormalizable with the same period $a\in\nt$. Denote by $\varepsilon>0$ the $C^2$ distance between $\zeta$ and $\tilde{\zeta}$, that is, $\varepsilon=d_2(\zeta,\tilde{\zeta})$. We may assume in the computations that $\varepsilon\in[0,1)$. We will only consider the special situation when
\begin{itemize}
\item[1)] $I_\eta=I_{\tilde{\eta}}$ and $I_\xi=I_{\tilde{\xi}}$,
\item[2)] $p=\tilde{p}$ where $D\eta(p)=D\tilde{\eta}(\tilde{p})=1$ (see Lemma \ref{op}).
\end{itemize}

For each $i\in\{1,...,a-1\}$ let $f_i\in\text{Diff}_+^3\big([0,1]\big)$ given by $f_i=A_{i+1}^{-1}\!\circ\! \eta\!\circ\! A_i$, where $A_i:[0,1] \to I_i$ is the unique orientation preserving affine diffeomorphism:$$A_i(x)=|I_i|x+x_i=
\big(\eta^{i-1}\big(\xi(0)\big)-\eta^{i}\big(\xi(0)\big)\big)x+\eta^{i}\big(\xi(0)\big)$$

Note that $\bigcirc_{i=1}^{a-1}f_i=A_{a}^{-1}\circ\eta^{a-1}\circ A_1$ in $\text{Diff}_+^3\big([0,1]\big)$.

\begin{lem}\label{NeDN} There exists $K(\mathcal{K})>1$ such that for any $\zeta$ in $\mathcal{K}$ renormalizable with period $a\in\nt$ we have:$$\sum_{i=1}^{a-1}\big|Nf_i(x)\big| \leq K\quad\mbox{and}\quad\sum_{i=1}^{a-1}\big|D(Nf_i)(x)\big| \leq K\quad\mbox{for all $x\in[0,1]$.}$$
\end{lem}

\begin{proof}[Proof of Lemma \ref{NeDN}] Note that $Nf_i(x)=N(\eta\circ A_i)(x)=N\eta\big(A_i(x)\big)|I_i|$ and that $D(Nf_i)(x)=D(N\eta)(A_i(x))|I_i|^2$ for all $x\in[0,1]$. Since $\zeta\in\mathcal{K}_a$ we know that $N\eta$ is $C^1$-bounded in $\big[\eta^{a}(\xi(0)),\xi(0)\big]$ (see Remark \ref{rememap} at the end of Section \ref{Seccont}) and then:$$\sum_{i=1}^{a-1}\big|Nf_i(x)\big| \leq K\sum_{i=1}^{a-1}|I_i|\leq K|I_{\eta}|\quad\mbox{and:}$$

$$\sum_{i=1}^{a-1}\big|D(Nf_i)(x)\big| \leq K\sum_{i=1}^{a-1}|I_i|^2\leq K|I_{\eta}|\sum_{i=1}^{a-1}|I_i| \leq K|I_{\eta}|^2.$$
\end{proof}

In the same way let $\tilde{A}_i:[0,1] \to \tilde{I}_i$ be the unique orientation preserving affine diffeomorphism, and define $g_i=\tilde{A}_{i+1}^{-1}\circ\tilde{\eta}\circ \tilde{A}_i\in\text{Diff}_+^3\big([0,1]\big)$.

\bigskip

The first factors of the renormalizations are controlled by

\begin{lem}\label{primestprop} There exists $K>0$ such that
$$
|A_{1}^{-1}\circ\xi|_{C^3}, | \tilde{A}_{1}^{-1}\circ\tilde{\xi}|_{C^3}\le K.
$$ and
$$d_{C^2}\big(A_{1}^{-1}\circ\xi,\tilde{A}_{1}^{-1}\circ\tilde{\xi}\big) \leq K\varepsilon.$$
\end{lem}

\begin{proof}[Proof of Lemma \ref{primestprop}] The four maps $\xi:[-1,0] \to I_1$, $\tilde{\xi}:[-1,0] \to \tilde{I}_1$, $A_{1}^{-1}:[0,K]\to\R$ and $\tilde{A}_{1}^{-1}:[0,K]\to\R$ are $C^3$-bounded by some constant $M>1$ universal on $\mathcal{K}$.  Similar to Lemma \ref{comp2} we get

\begin{align*}
d_{C^2}\big(A_{1}^{-1}\circ\xi,\tilde{A}_{1}^{-1}\circ\tilde{\xi}\big) &\leq K\left(\big\|A_{1}^{-1}-\tilde{A}_{1}^{-1}\big\|_{C^2}+\big\|\xi-\tilde{\xi}\big\|_{C^2}\right)\\
&\leq K\left(\big\|A_{1}^{-1}-\tilde{A}_{1}^{-1}\big\|_{C^2}+\varepsilon\right).
\end{align*}

Observe,

\begin{align*}
\big|A_{1}^{-1}(x)-\tilde{A}_{1}^{-1}(x)\big|&=\big||I_1|^{-1}(x-x_1)-|\tilde{I}_1|^{-1}(x-\tilde{x}_1)\big|\\
&=\frac{\big|(x-x_1)(\tilde{x}_0-\tilde{x}_1)-(x-\tilde{x}_1)(x_0-x_1)\big|}{|I_1||\tilde{I}_1|}\\
&=\frac{\big|x(\tilde{x}_0-x_0)+x(x_1-\tilde{x}_1)+(x_0\tilde{x}_1-\tilde{x}_0x_1)\big|}{|I_1||\tilde{I}_1|}\\
&\leq K\left(\frac{\Delta x_0+\Delta x_1+|x_0\tilde{x}_1-\tilde{x}_0x_1|}{|I_1||\tilde{I}_1|}\right)\\
&\leq K\left(\frac{\Delta x_0+\Delta x_1+|x_0||\tilde{x}_1-x_1|+|x_1||x_0-\tilde{x}_0|}{|I_1||\tilde{I}_1|}\right)\\
&\leq K(\Delta x_0+\Delta x_1)/|I_1||\tilde{I}_1|\leq K\varepsilon,
\end{align*}
where we used Lemma \ref{entrada}.

On the other hand 
$$\big|(A_{1}^{-1})'-(\tilde{A}_{1}^{-1})'\big|=\big(|\tilde{I}_1|-|I_1|\big)/|I_1||\tilde{I}_1|
\leq(\Delta_0+\Delta_1)/|I_1||\tilde{I}_1|,$$ and we finish in the same way as before.
\end{proof}

\begin{lem}\label{Dnfi} There exists $K>0$ such that for $i\le a$
$$
|Nf_i-Ng_i|_{C^0}\le K\big( \varepsilon |I_i|+ \Delta_i+ 
|\Delta x_i| |I_i|  \big).
$$
\end{lem}

\begin{proof}[Proof of Lemma \ref{Dnfi}] Observe,$$|\tilde{A}_i x-A_i x| \le K\big(|\Delta x_i|+\Delta_i\big).$$
So,
$$
\begin{aligned}
|Nf_i(x)-Ng_i(x)|=&|Nf(A_i(x))|I_i|-Ng(\tilde{A}_i(x))|\tilde{I}_i||\\
\le &\big| Nf(A_i x) |I_i|-Ng(A_i x)|\tilde{I}_i| \big|+\\
&|DNg(\theta_i)| \cdot ( |\Delta x_i| +\Delta_i )\cdot  |\tilde{I}_i|\\
\le & K\big( \varepsilon |I_i|+ \Delta_i+( |\Delta x_i| +\Delta_i )( |I_i| +\Delta_i)\big)\\
\le &K\big( \varepsilon |I_i|+ \Delta_i+ 
|\Delta x_i| |I_i|  \big).
\end{aligned}
$$
\end{proof}

\begin{lem}\label{sumDnfi}  For every $L>0$ there exists $K=K(\mathcal{K}, L)>0$ such that the following holds. If $\zeta$ and $\tilde{\zeta}$ are $L$-synchronized then
$$
\sum_{i=1}^a \big| Nf_i-Ng_i \big|_{C^0}\le K \varepsilon.
$$
\end{lem}

\begin{proof}[Proof of Lemma \ref{sumDnfi}] Let $a_\varepsilon=\left\lfloor\frac{1}{\varepsilon}\right\rfloor$. Assume for a moment that $a\ge a_\varepsilon$. Then Lemma \ref{aquele} implies $|x_{a-a_\varepsilon}-x_{a_\varepsilon}| ,
|\tilde{x}_{a-a_\varepsilon}-\tilde{x}_{a_\varepsilon}|\le K \varepsilon$.
Hence,
\begin{equation}\label{center}
\begin{aligned}
\sum_{a_\varepsilon\le i \le a -a_\varepsilon} \big| Nf_i-Ng_i \big|_{C^0}&\le \sum_{a_\varepsilon\le i \le a -a_\varepsilon} \big| Nf_i\big|_{C^0}+\big|Ng_i \big|_{C^0}\\
&\le \sum_{a_\varepsilon\le i \le a -a_\varepsilon} \big|
  Nf\big|_{C^0}\cdot |I_i|+\big|Ng\big|_{C^0}\cdot |\tilde{I}_i|\\
&\le K\big( |x_{a-a_\varepsilon}-x_{a_\varepsilon}| +|\tilde{x}_{a-a_\varepsilon}-\tilde{x}_{a_\varepsilon}|\big)\\
&\le K \varepsilon.
\end{aligned}
\end{equation}
This estimates holds trivially when $a<a_\varepsilon$.

Observe,
$$
\begin{aligned}
\sum_{i=1}^a \big| Nf_i-Ng_i \big|_{C^0}=&
\sum_{i=1}^{a_\varepsilon} \big| Nf_i-Ng_i \big|_{C^0}+\sum_{i=a_0}^{a-a_\varepsilon} \big| Nf_i-Ng_i \big|_{C^0}+\\
&\sum_{i=a-a_\varepsilon}^a \big| Nf_i-Ng_i \big|_{C^0}.
\end{aligned}
$$
The middle term is estimated by (\ref{center}). The first (and third) term can be estimated by using Lemma \ref{Dnfi}, the Yoccoz Lemma \ref{doyoc}, the Propositions \ref{Dxi} and \ref{Di}. Namely,
$$
\begin{aligned}
\sum_{i=1}^{a_\varepsilon} \big| Nf_i-Ng_i \big|_{C^0}&\le 
K
\sum_{i=1}^{a_\varepsilon} \varepsilon |I_i|+ \Delta_i+ 
|\Delta x_i| |I_i|  \\
&\le  
K \sum_{i=1}^{a_\varepsilon} \varepsilon \frac{1}{i^2}+  
\varepsilon \cdot \frac{ \log i}{i^2} +  \varepsilon^2
\cdot \frac{1}{i}+
\varepsilon \cdot \frac{1}{i^3} \\
&\le  K\varepsilon+K \sum_{i=1}^{a_\varepsilon}   \varepsilon^2
\cdot \frac{1}{i}\\
&\le  K\varepsilon+K\varepsilon^2 \log \frac{1}{ \varepsilon}\\
&\le  K\varepsilon.
\end{aligned}
$$
The Lemma follows.
\end{proof}

The following Proposition holds for general critical commuting pairs with negative Schwarzian which are contained in the previously discussed set $\mathcal{K}$, that is, the set of normalized $C^3$ critical commuting pairs which are $K$-controlled.

\begin{prop}\label{lipC2C2amigos}  For every $L>0$ there exists $K=K(\mathcal{K}, L)>0$ such that the following holds. If $\zeta_0$ and $\zeta_1$ are $L$-synchronized then
$$
d_2\big(p\mathcal{R}(\zeta_0),p\mathcal{R}(\zeta_1)\big)\le K  d_2(\zeta_0,\zeta_1).
$$
\end{prop}

\begin{proof}[Proof of Proposition \ref{lipC2C2amigos}] There exists $K=K(\mathcal{K})>0$ such that the following holds. There exists a diffeomorphism $h:\Dom(\zeta_1)\to \Dom(\zeta_0)$ such that $\zeta=\zeta_0$ and
$\tilde{\zeta}=h\circ\zeta_1\circ h^{-1}$ satisfy the normalizations 
\begin{itemize}
\item[1)] $I_\eta=I_{\tilde{\eta}}$ and $I_\xi=I_{\tilde{\xi}}$,
\item[2)] $p=\tilde{p}$ where $D\eta(p)=D\tilde{\eta}(\tilde{p})=1$,
\end{itemize}
needed to apply the results from section \S \ref{orbdef} and  \S \ref{comp}. We may 
construct the conjugation such that
$$
d_{C^3}(h, \id)\le K d_2(\zeta_0,\zeta_1)
$$
and 
$h|\Dom\big(p\mathcal{R}(\zeta_1)\big)=\id$. This last condition implies
$$
p\mathcal{R}(\zeta_1)=p\mathcal{R}(\tilde{\zeta}).
$$
In particular, it suffices to prove the Proposition for the pairs $\zeta$ and $\tilde{\zeta}$.

\bigskip

Let $p\mathcal{R}(\zeta)=(\eta', \xi')$ and $p\mathcal{R}(\tilde{\zeta})=(\tilde{\eta'}, \tilde{\xi'})$.
Because, $\xi'=\xi$ and $\tilde{\xi}'=\tilde{\xi}$ it suffices to estimate
the distance between $\eta'$ and $\tilde{\eta}'$.  

Let $I_{a+1}=[x_{a+1}, x_a]$ and $A:[0,1]\to I_{a+1}$ be the orientation preserving affine diffeomorphism. Let 
$$
F=A^{-1}\circ \eta',
$$ 
and similarly define
$G=\tilde{A}^{-1}\circ \tilde{\eta}'$. Now apply Lemma \ref{NLcomp} and Lemma \ref{sumDnfi} to obtain
$$
d_{C^2}(F,G)\le K\varepsilon,
$$
where $\varepsilon=d_2(\zeta,\tilde{\zeta})$. A similar argument as the proof of Lemma \ref{primestprop} one obtains
$
d_2(\eta', \tilde{\eta}')\le K\varepsilon.
$
This shows that prerenormalization is Lipschitz among synchronized pairs.
\end{proof}

\section{Order}\label{order}

Commuting pairs might have different domains. Any natural definition of {\it order}
between such systems has to include this difference of domains also. There are two cases:$$\mbox{case I: $\eta\circ \xi(0)>0$,}\quad\quad\quad\mbox{case II:  $\eta\circ \xi(0)<0$.}$$

\begin{defn}\label{ord} Let $\zeta_0=(\xi_0,\eta_0)$ and  
$\zeta_1=(\xi_1,\eta_1)$ be two commuting pairs and $t\ge 0$. If 
\begin{itemize}
\item[1)]  $
\zeta_0(x)+t\le \zeta_1(x),
$ for $x\in \Dom(\zeta_0)\cap \Dom(\zeta_1)$,
\item[2)] $\eta_0(0)\le \eta_1(0)$ and $\xi_0(0)\le \xi_1(0)$
\end{itemize}
we write
$$
\zeta_0\le_t \zeta_1.
$$
\end{defn}

\begin{lem}\label{orbitorder}  Let $\zeta_0=(\xi_0,\eta_0)$ and  
$\zeta_1=(\xi_1,\eta_1)$ be two commuting pairs. If 
$\zeta_0\le_t \zeta_1$ then

case I:
\begin{itemize}
\item[1)] $a_{\zeta_0}\le a_{\zeta_1}$,
\item[2)] for $x\in [\eta_1(0),0]$ and $k=0,1,\cdots, a_{\zeta_0}$
$$
\eta_0^k\circ \xi_0(x)+t\le \eta_1^k\circ \xi_1(x).
$$
\end{itemize}

case II:
\begin{itemize}
\item[1)] $a_{\zeta_0}\ge a_{\zeta_1}$,
\item[2)] for $x\in [0, \xi_0(0)]$ and $k=0,1,\cdots, a_{\zeta_1}$
$$
\xi_0^k\circ \eta_0(x)+t\le \xi_1^k\circ \eta_1(x).
$$
\end{itemize}
\end{lem}

The proof of Lemma \ref{orbitorder} is different for case I and case II. We will only present the proof in case I.

\begin{proof}[Proof of Lemma \ref{orbitorder}] As we said, we will only present the proof in case I. Let $x\in [0, \xi_0(0)]$. The order condition Definition \ref{ord}(1) gives the statement of the Lemma for $k=0$, 
 $\xi_0(x)+t\le \xi_1(x)$. Inductively property (2) follows. Namely,
$$
\begin{aligned}
\eta_0^{k+1}\circ \xi_0(x)+t&=\eta_0(\eta_0^k\circ \xi_0(x))+t
                            \le \eta_1(\eta_0^k\circ \xi_0(x))\\
                            &\le  \eta_1(\eta_1^k\circ \xi_1(x))\\
                            &= \eta_1^{k+1}\circ \xi_1(x). 
\end{aligned}
$$
In particular, 
$\eta_0^{a_{\zeta_0}}\circ \xi_0(x)\le \eta_1^{a_{\zeta_0}}\circ \xi_1(x)$. This implies,  $a_{\zeta_0}\le a_{\zeta_1}$.
\end{proof}

Pre-renormalization preserves order. Namely,

\begin{lem}\label{pRorder} If $\zeta_0\le_t \zeta_1$ and $a_{\zeta_0}=a_{\zeta_1}$, then $p\mathcal{R}(\zeta_0)\le_t p\mathcal{R}(\zeta_1)$.
\end{lem}

\begin{proof}[Proof of Lemma \ref{pRorder}] We will only present the proof in case I. Let 
$a=a_{\zeta_0}=a_{\zeta_1}$. Observe,
$\eta_{p\mathcal{R}(\zeta_0)}(0)=\eta_0(0)\le \eta_1(0)=\eta_{p\mathcal{R}(\zeta_1)}(0)$. Hence, the left side of the domains of the pre-renormalizations satisfy the order condition of Definition \ref{ord}(2). Consider the right side of the domains of the pre-renormalizations,
\begin{equation}\label{boundary}
\xi_{p\mathcal{R}(\zeta_0)}(0)+t=\eta_0^a\circ \xi_0(0)+t\le 
\eta_1^a\circ \xi_1(0)=\xi_{p\mathcal{R}(\zeta_1)}(0),
\end{equation}
where we used Lemma \ref{orbitorder}(2). This means that the right side of the domain of the pre-renormalizations also satisfy  the order condition of Definition \ref{ord}(2). 

According to Lemma \ref{orbitorder}(2) the estimate (\ref{boundary}) also hold 
for any $x\in [\eta_1(0),0]$, instead of $x=0$. This means that the pre-renormalization also satisfy the order condition of Definition \ref{ord}(1).
\end{proof}

The following Proposition will play a key role in the proof of 
the Synchronization-Lemma, section \S\ref{sync}.

\begin{prop}\label{orderandrot}  If $\zeta_0\le_t \zeta_1$ with $t>0$ then
$$
\rho_{\zeta_0}\ne \rho_{\zeta_1}.
$$
\end{prop}

\begin{proof}[Proof of Proposition \ref{orderandrot}] Assume $a_{\zeta_0}(n)=a_{\zeta_1}(n)$ for $n\ge 0$. 
Apply Lemma \ref{pRorder}, 
$$
(p\mathcal{R})^n(\zeta_0) \le_t (p\mathcal{R})^n(\zeta_1).
$$
Note, $\eta_{(p\mathcal{R})^n(\zeta_{0,1})}(0)\to 0$. Hence,
$$
0>\eta_{(p\mathcal{R})^n(\zeta_{1})}(0)\ge \eta_{(p\mathcal{R})^n(\zeta_{0})}(0)+t\ge \frac12 t>0
$$
for $n$ large enough. Contradiction.
\end{proof}

\section{Synchronization}\label{sync}

\noindent
{\bf Synchronization-Lemma.} For any given $K_0>1$ there exists $L=L(K_0)>1$ such that the following holds. {\it Let $\zeta_0$ and $\zeta_1$ be two $C^3$ critical commuting pairs which are $K_0$-controlled, both $\zeta_0$ and $\zeta_1$ have negative Schwarzian, $\rho(\zeta_0)=\rho(\zeta_1)\in[0,1]\!\setminus\!\Q$ and $d_2(\zeta_0,\zeta_1)<\varepsilon_0$. Then $\zeta_0$ and $\zeta_1$ are $L$-synchronized.}

\bigskip

The hypothesis $d_2(\zeta_0,\zeta_1)<\varepsilon_0$ will not be mentioned in the proof presented below, but it is needed in order to be allowed to apply the estimates obtained in Sections \ref{orbdef} to \ref{order} (see in particular the proof of Claim \ref{lacompdelmyk}, during the proof of Lemma \ref{entrada}).

\begin{proof}[Proof of the Synchronization-Lemma] We will only present the proof in case I. Let $a=a_{\zeta_0}=a_{\zeta_1}$. Choose $a_0\ge 1$ such that Lemma \ref{temamigo} applies. The Synchronization Lemma follows from Lemma \ref{bruteforce} when $a\le a_0$. We will assume $a\ge a_0$.

\bigskip

We may assume that $x^1_a\ge x^0_a$. There exists $K=K(K_0)>0$
such that the following holds: there exists a diffeomorphism $h:\Dom(\zeta_1)\to \Dom(\zeta_0)$ such that $\zeta=\zeta_0$ and
$\tilde{\zeta}=h\circ\zeta_1\circ h^{-1}$ satisfy the normalizations 
$$
x_1(\zeta)=x_1(\tilde{\zeta}).
$$
We may 
construct the conjugation such that
$$
d_{C^3}(h, \id)\le K d_2(\zeta_0,\zeta_1)
$$
and 
$h|\Dom\big(p\mathcal{R}(\zeta_1)\big)=\id$. This last condition implies
$$
x_a(\zeta_1)=x_a(\tilde{\zeta}).
$$
In particular, it suffices to prove synchronization for the pairs
$\zeta$ and $\tilde{\zeta}$. Let $\varepsilon=
d_2(\zeta,\tilde{\zeta})\le  K d_2(\zeta_0,\zeta_1)$.

Apply Lemma \ref{temamigo} to obtain a commuting pair $\zeta_{t_0}$ in the standard family of $\zeta$ such that 
$$
\Delta x_a(\zeta_{t_0},\tilde{\zeta})=0.
$$
From Lemma \ref{temamigo} we get
\begin{equation}\label{t0}
0\le t_0\le K \varepsilon.
\end{equation}
Note, if $t_0>0$ is much larger than $\varepsilon\ge d_{C^0}(\zeta,\tilde{\zeta})$  then
$\xi_{t_0}(x)>\tilde{\xi}(x)$. This would imply
$x_a(\zeta_{t_0})>x_a(\tilde{\zeta})$ because $x_1(\zeta)=x_1(\tilde{\zeta})$. Assume that
\begin{equation}\label{notsync}
\tilde{x}_a=x_a+L\varepsilon,
\end{equation}
where just as before $x_i=\eta^{i}(\xi(0))$ and $\tilde{x}_i=\tilde{\eta}^{i}(\tilde{\xi}(0))$ for $i\in\{0,...,a\}$. Note also that the assumption $x_{a}^{1} \geq x_{a}^{0}$ implies that $\tilde{x}_a \geq x_a$.

We have to show that $L$ is uniformly bounded.

From (\ref{notsync}) and Corollary \ref{MFprimcor} we get for every 
$x\in [\eta_{p\mathcal{R}(\zeta_{t_0})}(0),0]$
\begin{equation}\label{zetat0zeta}
\begin{aligned}
 p\mathcal{R}(\zeta_{t_0})(x)- p\mathcal{R}(\zeta)(x) &\ge \frac{1}{K}  
\big(p\mathcal{R}(\zeta_{t_0})(0)- p\mathcal{R}(\zeta)(0)\big)\\
&=\frac{1}{K}  \big(
 p\mathcal{R}(\tilde{\zeta})(0)- p\mathcal{R}(\zeta)(0)\big)\\
&=\frac{1}{K} (\tilde{x}_a-x_a)=\frac{1}{K} L\varepsilon.
\end{aligned}
\end{equation}
From Proposition \ref{lipC2C2amigos} we get  for every 
$x\in [\eta_{p\mathcal{R}(\zeta_{t_0})}(0),0]$
\begin{equation}\label{zetat0zetatilde}
\begin{aligned}
\big|p\mathcal{R}(\zeta_{t_0})(x)- p\mathcal{R}(\tilde{\zeta})(x)\big| &\le K 
d_2(\zeta_{t_0}, \tilde{\zeta})\\
&\le K d_2(\zeta, \tilde{\zeta}) +K\varepsilon\\
&\le K\varepsilon,
\end{aligned}
\end{equation}
where we also used (\ref{t0}). Combine (\ref{zetat0zeta})
and (\ref{zetat0zetatilde}) to get for every $x\in [\eta_{p\mathcal{R}(\zeta_{t_0})}(0),0]$
\begin{equation}\label{dpRzeta10}
p\mathcal{R}(\tilde{\zeta})(x)\ge p\mathcal{R}(\zeta)(x)+\frac{1}{K} L\varepsilon -K\varepsilon.
\end{equation}
As a matter of fact (\ref{dpRzeta10}) holds for $x\in [-1,0]$. This
follows from the following. Let $x\in[-1,\eta_{p\mathcal{R}(\zeta_{t_0})}(0)]$. 
Observe, according to (\ref{t0}),$$\big|[-1,\eta_{p\mathcal{R}(\zeta_{t_0})}(0)]\big|=t_0\le K\varepsilon.$$ 
This implies
$$
\begin{aligned}
p\mathcal{R}(\tilde{\zeta})(x)&\ge p\mathcal{R}(\zeta)(\eta_{p\mathcal{R}(\zeta_{t_0})}(0))
+\frac{1}{K} L\varepsilon -K\varepsilon
-\max\{Dp\mathcal{R}(\tilde{\zeta})\} t_0\\
&\ge p\mathcal{R}(\zeta)(x)+\frac{1}{K} L\varepsilon -K\varepsilon.
\end{aligned}
$$
Hence, for $x\in [-1,0]$ we have
\begin{equation}\label{dpRzeta10total}
p\mathcal{R}(\tilde{\zeta})(x)\ge p\mathcal{R}(\zeta)(x)+\frac{1}{K} L\varepsilon -K\varepsilon.
\end{equation}
So, when $L\ge 2K^2$ then for the relevant $x<0$
\begin{equation}\label{dpR2}
(p\mathcal{R})^2(\tilde{\zeta})(x)> (p\mathcal{R})^2(\zeta)(x).
\end{equation}
The last part of the proof will show that similar estimates hold for 
relevant positive points. The goal is to prove 
$\big(p\mathcal{R}\big)^2(\tilde{\zeta})\ge_t \big(p\mathcal{R}\big)^2(\zeta)$ for some positive $t$. The branches  on the left side of the second pre-renormalizations, according to (\ref{dpR2}), 
satisfy the order condition of Definition \ref{ord}(1).
The right side of the domains of the second pre-renormalizations do satisfy the order condition of Definition \ref{ord}(2). Namely, 
$$
\Dom\big((p\mathcal{R})^2(\zeta)\big)\cap \{x\ge 0\}=[0, x_a]\subset
[0, \tilde{x}_a]=
 \Dom\big((p\mathcal{R})^2(\tilde{\zeta})\big)\cap \{x\ge 0\}.
$$
Left is to describe the branches on the right and the domains on the left. Let $x\in \Dom\big((p\mathcal{R})^2(\zeta)\big)\cap \{x\ge 0\}=[0, x_a]$ and for $k\ge 1$ define
$$
z_k(x)=\big(p\mathcal{R}(\zeta)\big)^k(x),
$$
and similarly, $
\tilde{z}_k(x)=\big(p\mathcal{R}(\tilde{\zeta})\big)^k(x),
$
Observe,
$$
|z_1(x)-\tilde{z}_1(x)|=|p\mathcal{R}(\zeta)(x)- p\mathcal{R}(\tilde{\zeta})(x)|=|\eta(x)-\tilde{\eta}(x)|\le \varepsilon.
$$
Hence, applying (\ref{dpRzeta10total}),
$$
\begin{aligned}
\tilde{z}_2(x)&=p\mathcal{R}(\tilde{\zeta})(\tilde{z}_1)\\
&\ge p\mathcal{R}(\zeta)(\tilde{z}_1)
+\frac{1}{K} L\varepsilon -K\varepsilon\\
&\ge z_2(x) -\max\big(Dp\mathcal{R}(\zeta)\big)\cdot |z_1(x)-\tilde{z}_1(x)|
+\frac{1}{K} L\varepsilon -K\varepsilon\\
&\ge z_2(x) 
+\frac{1}{K} L\varepsilon -K\varepsilon> z_2(x) ,
\end{aligned}
$$
when $L\ge 2K^2$. Let $b=a_{p\mathcal{R}(\zeta)}=a_{p\mathcal{R}(\tilde{\zeta})}$. By repeatedly applying 
(\ref{dpRzeta10total}) with $L\ge 2K^2$ we obtain
\begin{equation}\label{rightbranches}
\big(p\mathcal{R}\big)^2(\tilde{\zeta})(x)=\tilde{z}_b(x)>\big(p\mathcal{R}\big)^2(\zeta)(x)=z_b(x).
\end{equation}
In particular,
\begin{equation}\label{leftboundary}
\begin{aligned}
\Dom\big((p\mathcal{R})^2(\tilde{\zeta})\big)\cap \{x\le 0\}&=[\tilde{z}_b,0]\subset [z_b,0]\\
&=\Dom\big((p\mathcal{R})^2(\zeta)\big)\cap \{x\le 0\}.
\end{aligned}
\end{equation}

The estimates (\ref{rightbranches}) and (\ref{leftboundary}) finish the proof of:$$\big(p\mathcal{R}\big)^2(\tilde{\zeta})\ge_t \big(p\mathcal{R}\big)^2(\zeta),$$for some  $t>0$. However, this contradicts Proposition \ref{orderandrot} because $\big(p\mathcal{R}\big)^2(\tilde{\zeta})$ and $\big(p\mathcal{R}\big)^2(\zeta)$ have the same rotation number. This contradiction establishes the synchronization with $L\le 2K^2$. 
\end{proof}

\section{Lipschitz Estimate}\label{lip}

In this section we prove Lemma \ref{main}.

\begin{proof}[Proof of Lemma \ref{main}] According to the Synchronization Lemma from section \S\ref{sync} we know that  for $L=L(\mathcal{K})$, the pairs $\zeta_0$ and $\zeta_1$ are $L$-synchronized. Now the Lipschitz estimate for renormalization of synchronized pairs, Proposition \ref{lipC2C2amigos}, imply a Lipschitz estimate for prerenormalization along topological classes. The fact that the maps are synchronized imply that the domains of the prerenormalizations are also close. This means that the normalizations will not effect the Lipschitz property. 
\end{proof}

\section{The attractor of renormalization}\label{melhoratese}

This section is devoted to the following result:

\begin{theorem}\label{compacto} There exists a $C^{\omega}$-compact set $\mathcal{K}$ of real analytic critical commuting pairs with the following property: for any $r \geq 3$ there exists a constant $\lambda=\lambda(r) \in (0,1)$ such that given a $C^r$ critical circle map $f$ with irrational rotation number there exist $C>0$ and a sequence $\{f_n\}_{n\in\nt}$ contained in $\mathcal{K}$ such that:$$d_{r-1}\big(\mathcal{R}^n(f),f_n\big)\leq C\lambda^n\quad\mbox{for all $n\in\nt$,}$$and such that the pair $f_n$ has the same rotation number as the pair $\mathcal{R}^n(f)$ for all $n\in\nt$. Here $d_{r-1}$ denotes the $C^{r-1}$ distance in the space of $C^{r-1}$ critical commuting pairs.
\end{theorem}

We will apply Theorem \ref{compacto} in the next section (Section \ref{final}, during the proof of Theorem \ref{expconv}) with $r=4$.

The compact set $\mathcal{K}$ and the approximations $\{f_n\}_{n\in\nt}$ given by Theorem \ref{compacto} were constructed by two of the authors in \cite{guamelo}, but the exponential convergence was only proved for the $C^0$-metric \cite[Theorem D, Section 4, page 15]{guamelo}. In this section we will show that the same estimate actually holds for the $C^{r-1}$-metric, whenever $f$ is $C^r$. For that purposes we will use the following fact from complex analysis:

\begin{prop}\label{holder} Let $I$ be a compact interval in the real line with non-empty interior, and let $U$ be an open set in the complex plane containing $I$. Fix $M>0$ and consider the family:$$\mathcal{F}=\big\{f:U\to\C\mbox{ holomorphic: }\|f\|_{C^0(U)} \leq M\big\}.$$

Then for any $r\in\nt$ and any $\alpha\in(0,1)$ there exists $L(r,\alpha,M)>0$ such that:$$\|f\|_{C^r(I)} \leq L\big(\|f\|_{C^0(I)}\big)^\alpha\quad\mbox{for all $f\in\mathcal{F}$,}$$where:$$\|f\|_{C^r(I)}=\sup_{\substack{z\in I\\n\in\{0,1,...,r\}}}\big\{\big|f^{(n)}(z)\big|\big\}.$$
\end{prop}

\begin{remark}\label{tg} Let us mention that both components of each approximation $f_n$ constructed in \cite{guamelo} have holomorphic extensions satisfying the conditions of Proposition \ref{holder} (uniformly bounded on a definite domain), see \cite[Section 7]{guamelo} for the construction.
\end{remark}

In this section we explain how Theorem \ref{compacto} follows by combining \cite[Appendix A]{dfdm1} and \cite[Theorem D]{guamelo} with Proposition \ref{holder}. We postpone the proof of Proposition \ref{holder} until Appendix \ref{apA}. In the remainder of this section we assume, to simplify the exposition, that the criticality of the critical point is $3$, that is, $d=1$ in Condition \eqref{pc} in Definition \ref{critpair}.

\begin{definition}\label{eps} Let $I=[0,a]$ and let $g:I \to g(I)=J$ be a real analytic orientation preserving homeomorphism with a cubic critical point at $0$. We say that $g$ is an \emph{Epstein} map if there exist a topological disk $U \supset I$, an open interval $L \supset J$ and a holomorphic three-fold branched covering map $G:U\to\C_L$ such that $G|_{I}=g$ (as usual, $\C_L$ denotes the open set $\C\!\setminus\!(\R\!\setminus\! L)$).
\end{definition}

For any $\beta \in (0,1)$ denote by $\mathcal{E}_{\beta}$ the set of Epstein maps $g:I=[0,a] \to g(I)=J$ satisfying the following properties:

\begin{enumerate}
\item $\beta\leq|I|/|J|\leq\beta^{-1}$.
\item $\dist(I,J)\leq\beta^{-1}|J|$, where $\dist$ denotes the standard distance between compact sets in the real line.
\item $g'(a)>\beta$.
\item The length of each component of $L\!\setminus\!J$ is at least $\beta|J|$ and at most $\beta^{-1}|J|$.
\end{enumerate}

In order to apply Proposition \ref{holder} we will need the following fact:

\begin{prop}\label{epscomp} For any $\beta\in(0,1)$ there exist a Jordan domain $U_{\beta}\supset[0,1]$ and a positive constant $M_{\beta}$ such that for any $g\in\mathcal{E}_{\beta}$, with normalization $I=[0,1]$, the holomorphic extension $G$ given by Definition \ref{eps} is well-defined in $U_{\beta}$ and satisfies $\big|G(z)\big| \leq M_{\beta}$ for all $z \in U_{\beta}$.
\end{prop}

\begin{proof}[Proof of Proposition \ref{epscomp}] Note that it is enough to prove the result for any sequence in $\mathcal{E}_{\beta}$. Let $g_n:I=[0,1]\to J_n\subset L_n$ be a normalized sequence in the class $\mathcal{E}_{\beta}$. By Definition \ref{eps} these maps extend to triple branched coverings $g_n:U_n\to\C_{L_n}$, where $U_n$ is a topological disk. Therefore, each $g_n$ can be decomposed as $g_n=Q_{c_n}\!\circ h_n$, where $Q_{c_n}(z)=z^3+c_n$ and $h_n:[0,1]\to[0,b_n]$ with $h_n\big([0,1]\big)=[0,b_n]$ is a univalent map $h_n:U_n\to\Omega(h_n)$ onto the complex plane with six slits, which triply covers $\C_{L_n}$ under $Q_{c_n}$. With this notation $J_n=[c_n,c_n+b_n^3]$, in particular $b_n^3=|J_n|\leq|I|/\beta$ and then the sequence of positive numbers $\big\{b_n=h_n(1)\big\}$ is bounded. Moreover, since the length of each component of $L_n\!\setminus\!J_n$ is at least $\beta|J_n|\geq\beta^2|I|$, there exists $\varepsilon>0$ such that $\overline{B(w,\varepsilon)}\subset\Omega(h_n)$ for all $w\in[0,b_n]$ and for all $n\in\nt$.

\begin{claim}\label{interno} There exists $\lambda>0$ such that $\big|\big(h_n^{-1}\big)'(w)\big|>\lambda$ for all $w\in[0,b_n]$ and for all $n\in\nt$.
\end{claim}

\begin{proof}[Proof of Claim \ref{interno}] We claim first that, by passing to a subsequence if necessary, the sequence of marked domains $(\Omega(h_n),0)$ converges to some marked domain $(\Omega,0)$ in the Carath\'eodory topology (for the definition of the Carath\'eodory topology, see the book of McMullen \cite[Chapter 5]{mclivro1}).

Indeed, note that $\dist\big([0,1],J_n\big)\leq|J_n|/\beta\leq|I|/\beta^2$, and then the intervals $J_n$ accumulate on an interval $J$. Moreover, since $|J_n|\geq\beta|I|$, the interval $J$ has non-empty interior. Since the length of each component of $L_n\!\setminus\!J_n$ is at least $\beta|J_n|\geq\beta^2|I|$, the intervals $L_n$ accumulate on an open interval $L$ that contains $J$. Moreover, the length of $L$ is finite since the length of each component of $L_n\!\setminus\!J_n$ is at most $\beta^{-1}|J_n|\leq\beta^{-2}|I|$. Since $c_n$ is the left boundary point of $J_n$, and $\dist\big([0,1],J_n\big)\leq|I|/\beta^2$, the sequence $\{c_n\}$ is bounded, and then (by passing to a subsequence if necessary) the sequence $(\Omega(h_n),0)$ converges to some $(\Omega,0)$ in the Carath\'eodory topology.

Secondly, we claim that the sequence of biholomorphisms $\big\{h_n^{-1}:\Omega(h_n) \to U_n\big\}$ is normal in $\Omega$ (note that any compact subset of $\Omega$ is eventually contained in $\Omega(h_n)$, and then $h_n$ is well-defined on it, again see \cite[Chapter 5]{mclivro1} for more on the Carath\'eodory topology). Indeed, from $Q_{c_n}=g_n \circ h_n^{-1}$ we get:$$\big(h_n^{-1}\big)'(b_n)=\frac{Q_{c_n}'(b_n)}{g_n'(1)}=\frac{3b_n^2}{g_n'(1)}\,,$$and then $\big|\big(h_n^{-1}\big)'(b_n)\big|<(3/\beta)|b_n|^2$ is bounded (since $b_n$ is bounded, as we already have seen). As we said before, since the length of each component of $L_n\!\setminus\!J_n$ is at least $\beta|J_n|\geq\beta^2|I|$, the points $b_n$ stay away from the boundary of $\Omega$, that is, $\inf_{n\in\nt}\big\{d\big([0,b_n],\partial\Omega\big)\big\}\geq\varepsilon>0$. Koebe Distortion Theorem (Theorem \ref{koebehol}) implies then that the family $\{h_n^{-1}\}$ is normal in $\Omega$. Since $b_n$ is bounded from above, any limit function is non-constant and therefore univalent. In particular it has no critical points, and this completes the proof of Claim \ref{interno}.
\end{proof}

With Claim \ref{interno} at hand and Koebe's one-quarter theorem (Theorem \ref{oneq}) we obtain:$$B\big(h_n^{-1}(w),\lambda\varepsilon/4\big)\subset\big(h_n^{-1}\big)\big(B(w,\varepsilon)\big)$$for all $w\in[0,b_n]$ and for all $n\in\nt$. Let $b=\sup_{n\in\nt}\{b_n\}$, and consider the two bounded Jordan domains:$$U_{\beta}=\bigcup_{z\in[0,1]}B(z,\lambda\varepsilon/4) \quad\mbox{and}\quad V_{\beta}=\bigcup_{w\in[0,b]}B(w,\varepsilon)\,.$$

For each $n\in\nt$ we have seen that $h_n$ is well-defined in $U_{\beta}$ and satisfies $h_n(U_{\beta})\subset V_{\beta}$. Since the sequence $\{c_n\}$ is bounded (as we pointed out before, in the proof of Claim \ref{interno}), each $g_n$ is (well-defined and) uniformly bounded in $U_{\beta}$. This completes the proof of Proposition \ref{epscomp}.
\end{proof}

\subsection{Proof of Theorem \ref{compacto}} The proof of Theorem \ref{compacto} given below relies on the following result of de Faria and de Melo \cite[Theorem A.6, Appendix A, page 382]{dfdm1}:

\begin{theorem}\label{ApC2} There exist $\beta\in(0,1)$ and $\lambda\in(0,1)$ with the following property: given any $C^r$ critical circle map $f$ with irrational rotation number, $r \geq 3$, there exists $C>0$ such that for each $n\in\nt$ there exist $\eta_n$ and $\xi_n$ in $\mathcal{E}_{\beta}$ such that $\xi_n$ has the same domain as $\tilde{f}^{q_{n}}$, $\eta_n$ has the same domain as $\tilde{f}^{q_{n+1}}$ and moreover:$$\big\|\xi_n-\tilde{f}^{q_n}\big\|_{C^{r-1}([-1,0])} \leq C\lambda^n\quad\mbox{and}\quad\big\|\eta_n-\tilde{f}^{q_{n+1}}\big\|_{C^{r-1}([0,|I_n|/|I_{n+1}|])} \leq C\lambda^n\,.$$
\end{theorem}

Unfortunately the pair $(\eta_n,\xi_n)$ given by Theorem \ref{ApC2} is \emph{not} a \emph{commuting} pair in general. In particular we have no information on the behaviour of the renormalization operator acting on these pairs.

\begin{proof}[Proof of Theorem \ref{compacto}] By \cite[Theorem D]{guamelo} there exists a $C^{\omega}$-compact set $\mathcal{K}$ of real analytic critical commuting pairs and a constant $\lambda_0\in(0,1)$ with the following property: given a $C^r$ critical circle map $f$, $r \geq 3$, with any irrational rotation number there exist $C_0>0$ and a sequence $\{f_n\}_{n\in\nt}$ contained in $\mathcal{K}$ such that:$$d_{0}\big(\mathcal{R}^n(f),f_n\big)\leq C_0\lambda_0^n\quad\mbox{for all $n\in\nt$,}$$and such that the pair $f_n$ has the same rotation number as the pair $\mathcal{R}^n(f)$ for all $n\in\nt$. From Theorem \ref{ApC2} we obtain $C_1>0$, $\lambda_1\in(0,1)$ and a sequence $\big\{(\eta_n,\xi_n)\big\}_{n\in\nt}$ such that:$$d_{r-1}\big(\mathcal{R}^n(f),(\eta_n,\xi_n)\big)\leq C_1\lambda_1^n\quad\mbox{for all $n\in\nt$.}$$

By considering $C_2=2\max\{C_0,C_1\}>0$ and $\lambda_2=\max\{\lambda_0,\lambda_1\}\in(0,1)$ we get that:$$d_{0}\big(f_n,(\eta_n,\xi_n)\big)\leq C_2\lambda_2^n\quad\mbox{for all $n\in\nt$.}$$

By Proposition \ref{holder}, Remark \ref{tg} and Proposition \ref{epscomp} the $C^{r-1}$-metric is H\"older equivalent to the $C^0$-metric for the sequences $\{f_n\}$ and $\big\{(\eta_n,\xi_n)\big\}$. More precisely:$$d_{r-1}\big(f_n,(\eta_n,\xi_n)\big)\leq C_3\lambda_3^n\quad\mbox{for all $n\in\nt$,}$$where $C_3=L\,C_2^{\alpha}>0$ and $\lambda_3=\lambda_2^{\alpha}\in(0,1)$, and the constants $L>0$ and $\alpha\in(0,1)$ are given by Proposition \ref{holder}. With this at hand we finally obtain, by the triangle inequality, that:$$d_{r-1}\big(\mathcal{R}^n(f),f_n\big)\leq C_4\lambda_4^n\quad\mbox{for all $n\in\nt$,}$$where $C_4=2\max\{C_1,C_3\}>0$ and $\lambda_4=\max\{\lambda_1,\lambda_3\}\in(0,1)$.
\end{proof}

\section{Exponential Convergence}\label{final}

In this section we prove the uniform exponential convergence of renormalization in the $C^4$ category (more precisely, we will prove that Theorem \ref{compacto} and Lemma \ref{main} combined with Theorem \ref{uniform} imply Theorem \ref{expconv}).

\begin{proof}[Proof of Theorem \ref{expconv}] Let $f$ and $g$ be two $C^4$ critical circle maps with the same irrational rotation number $\rho(f)=\rho(g)=[a_0,a_1,...]$ and with the same order at their critical points. By Theorem \ref{compacto} there exist a $C^{\omega}$-compact set $\mathcal{K}_0$ of real analytic critical commuting pairs, two constants $\lambda_0\in(0,1)$ and $C_0>1$ and two sequences $\{f_m\}_{m\in\nt}$ and $\{g_m\}_{m\in\nt}$ contained in $\mathcal{K}_0$ such that for all $m\in\nt$ we have $\rho(f_m)=\rho(g_m)=[a_m,a_{m+1},...]$ and:
\begin{equation}\label{perto}
d_3\big(\mathcal{R}^m(f),f_m\big)\leq C_0\,\lambda_0^m\quad\mbox{and}\quad d_3\big(\mathcal{R}^m(g),g_m\big)\leq C_0\,\lambda_0^m\,.
\end{equation}

From \cite[Theorem A.4, page 379]{dfdm1} we know that there exists $n_0\in\nt$ such that both critical commuting pairs $\mathcal{R}^n(f)$ and $\mathcal{R}^n(g)$ have negative Schwarzian bounded away from zero for all $n \geq n_0$. From Theorem \ref{compacto} we also know that the closure for the $C^3$-metric of the orbit $\big\{\mathcal{R}^n(f)\big\}_{n \geq n_0}$ is a $C^3$-compact set that we denote by $\mathcal{K}_f$ (the $\omega$-limit for the $C^3$-metric under renormalization of $f$ is contained in $\mathcal{K}_0$, which is $C^3$-compact). Let $\mathcal{K}_g$ be the corresponding compact set for $g$, that is, $\mathcal{K}_g$ is the closure for the $C^3$-metric of the orbit $\big\{\mathcal{R}^n(g)\big\}_{n \geq n_0}$. By compactness, there exists $\beta\in(0,1)$ with the following property: any $C^3$ critical commuting pair $\zeta$ such that $d_{C^3}(\zeta,\mathcal{K}_f\cup\mathcal{K}_g)<\beta$ has negative Schwarzian.

From the real bounds (Theorem \ref{realB}) there exist a universal constant $K_0>1$ and $n_1=n_1(f,g)\in\nt$, with $n_1>n_0$, such that the critical commuting pairs $\mathcal{R}^n(f)$ and $\mathcal{R}^n(g)$ are $K_0$-controlled for any $n \geq n_1$.

Let $n_2=n_2(2K_0)\in\nt$ be given by Theorem \ref{realBcomSneg}, and let $K=B^{n_2}(2K_0)$ be given by Lemma \ref{fromC2toC3} (here, the power $n_2$ denotes iteration). Let $\varepsilon_0(K)\in(0,1)$ and $L(K)>1$ be given by Lemma \ref{main}.

Fix $\delta\in(0,1)$ such that $\delta>\frac{\log L}{\log L-\log\lambda_0}$. Let $\lambda_2=L^{1-\delta}\lambda_0^{\delta}$, and note that $\lambda_2\in(0,1)$ since $\delta\log\lambda_0+(1-\delta)\log L<0$. For each $n\in\nt$ let $m=m(n)\in\nt$ given by $m=\lfloor\delta n\rfloor$, and fix $m_0\in\nt$ such that:$$m_0>\max\left\{\frac{n_2\log L+\log C_0+\log(1/\varepsilon_0)}{\log(1/\lambda_0)},\frac{\log C_0+\log(1/\beta)}{\log(1/\lambda_0)},n_1\right\}.$$

From \eqref{perto} we see that $d_2\big(\mathcal{R}^m(f),f_m\big)<\varepsilon_0$ and $d_3\big(\mathcal{R}^m(f),f_m\big)<\beta$ for all $m>m_0$, and also that $d_2\big(\mathcal{R}^m(g),g_m\big)<\varepsilon_0$ and $d_3\big(\mathcal{R}^m(g),g_m\big)<\beta$ for all $m>m_0$. In particular both critical commuting pairs $f_m$ and $g_m$ are $2K_0$-controlled for all $m \geq m_0$ and, moreover, both $f_m$ and $g_m$ have negative Schwarzian for all $m \geq m_0$ (and then the pairs $\mathcal{R}^j(f_m)$ and $\mathcal{R}^j(g_m)$ have negative Schwarzian for all $m \geq m_0$ and all $j\in\nt$, see Remark \ref{chainsch}).

By Theorem \ref{realBcomSneg} the critical commuting pair $\mathcal{R}^{j}(f_m)$ is $K_0$-controlled for all $m>m_0$ and $j \geq n_2$. For $j\in\{0,1,...,n_2\}$ we combine Lemma \ref{main} with \eqref{perto} to obtain that for all $m>m_0$:$$d_2\big(\mathcal{R}^{j}(f_m),\mathcal{R}^{j+m}(f)\big)\leq L^{j}d_2\big(f_m,\mathcal{R}^{m}(f)\big)\leq L^jC_0\lambda_0^{m}\leq L^{n_2}C_0\lambda_0^{m}<\varepsilon_0\,.$$

In particular, $\mathcal{R}^{j}(f_m)$ is $C^2$-bounded by $2K_0$ for all $m>m_0$ and $j\in\{0,1,...,n_2\}$, and therefore
$\mathcal{R}^{j}(f_m)$ is $K$-controlled for all $m>m_0$ and $j\in\{0,1,...,n_2\}$ by Lemma \ref{fromC2toC3}. This allows us to combine Lemma \ref{main} with \eqref{perto} in order to obtain:
\begin{align}\label{final1}
d_2\big(\mathcal{R}^n(f),\mathcal{R}^{n-m}(f_m)\big)&\leq L^{n-m}\cdot 	d_2\big(\mathcal{R}^m(f),f_m\big)\leq C_0\,L^{n-m}\lambda_{0}^{m}\leq\left(\frac{LC_0}{\lambda_0}\right)\lambda_2^n
\end{align}
for all $n\in\nt$ such that $m=\lfloor\delta n\rfloor>m_0$, since $L^{n-m-1}\lambda_0^{m+1}\leq(L^{1-\delta}\lambda_0^{\delta})^{n}=\lambda_2^n$. Let $C_3=LC_0/\lambda_0$. Replacing $f$ with $g$ we also get:
\begin{equation}\label{final2}
d_2\big(\mathcal{R}^n(g),\mathcal{R}^{n-m}(g_m)\big) \leq C_3\,\lambda_2^n\quad\mbox{for all $n\in\nt$ such that $m=\lfloor\delta n\rfloor>m_0$.}
\end{equation}

Since $f_m$ and $g_m$ are real analytic and have the same combinatorics for each $m\in\nt$, we know by Theorem \ref{uniform} that there exist constants $C_1>1$ and $\lambda_1 \in (0,1)$ (both uniform in $\mathcal{K}_0$) such that:
\begin{equation}\label{final3}
d_2\big(\mathcal{R}^{n-m}(f_m),\mathcal{R}^{n-m}(g_m)\big)\leq C_1\lambda_1^{n-m}\leq C_1(\lambda_1^{1-\delta})^{n}
\end{equation}
for all $n\in\nt$ such that $m=\lfloor\delta n\rfloor>m_0$. Finally we define $\lambda=\max\{\lambda_1^{1-\delta},\lambda_2\}=\max\{\lambda_1^{1-\delta},L^{1-\delta}\lambda_0^{\delta}\}\in(0,1)$ and $C=C_1+2C_3=C_1+2LC_0/\lambda_0>1$. Combining \eqref{final1}, \eqref{final2} and \eqref{final3} we get:$$d_2\big(\mathcal{R}^n(f),\mathcal{R}^n(g)\big) \leq C\lambda^n\quad\mbox{for all $n\in\nt$ such that $m=\lfloor\delta n\rfloor>m_0$.}$$
\end{proof}

\section{Rigidity}\label{finalmesmo}

As we said in the introduction, the fact that Theorem \ref{expconv} implies Theorem \ref{rigidity} follows from well-known results by de Faria-de Melo \cite[First Main Theorem, page 341]{dfdm1} and Khanin-Teplinsky \cite[Theorem 2, page 198]{khaninteplinsky}. In this final section we just give more precise references.

Let $f$ and $g$ be two $C^4$ circle homeomorphisms with the same irrational rotation number and with a unique critical point of the same odd type. Let $h$ be the unique topological conjugacy between $f$ and $g$ that maps the critical point of $f$ to the critical point of $g$. Let $\{\mathcal{P}^f_n\}_{n \geq 1}$ and $\{\mathcal{P}^g_n\}_{n \geq 1}$ be the corresponding sequences of dynamical partitions (see Section \ref{S:realbounds} of this paper), and note that the homeomorphism $h$ identifies those partitions.

In \cite[Section 3]{khaninteplinsky}, Khanin and Teplinsky proved that Theorem \ref{expconv} implies the existence of two constants $\hat C>0$ and $\hat\lambda \in (0,1)$ such that if $I_f,J_f$ are adjacent atoms in $\mathcal{P}^f_n$, or they are contained in the same atom of $\mathcal{P}^f_{n-1}$, and if $I_g=h(I_f),J_g=h(J_f)$ are the corresponding atoms in $\mathcal{P}^g_n$, we have that:
\begin{equation}\label{coherencekt}
\left|\log\frac{|I_g|}{|I_f|}-\log\frac{|J_g|}{|J_f|}\right|\leq\hat C\hat\lambda^n\quad\mbox{for all $n \geq 1$.}
\end{equation}

Combining these estimates with the real bounds, it is not difficult to prove the first two conclusions of Theorem \ref{rigidity}. See \cite[Proposition 1, page 199]{khaninteplinsky} for Conclusion \eqref{Aitem1}, and \cite[Remark 5, page 213]{khaninteplinsky} for Conclusion \eqref{Aitem2}.

To prove Conclusion \eqref{Aitem3} of Theorem \ref{rigidity}, however, it is not enough to have \eqref{coherencekt} for the dynamical partitions (indeed, note that \eqref{coherencekt} holds for \'Avila's examples \cite{avila} already mentioned in the introduction).

In \cite[Section 4]{dfdm1}, de Faria and de Melo constructed suitable partitions $\{\mathcal{Q}^f_n\}_{n \geq 1}$ and $\{\mathcal{Q}^g_n\}_{n \geq 1}$ (the so-called \emph{fine grids}, see \cite[Sections 4.2 and 4.3]{dfdm1}) such that for a full Lebesgue measure set of rotation numbers (see \cite[Section 4.4]{dfdm1} for its definition) Theorem \ref{expconv} implies the existence of two constants $\tilde C>0$ and $\tilde\lambda \in (0,1)$ such that:
\begin{equation}\label{coherencedfdm}
\left|\frac{|I_f|}{|J_f|}-\frac{|I_g|}{|J_g|}\right| \leq \tilde C\tilde\lambda^n\quad\mbox{for all $n \geq 1$.}
\end{equation}for each pair of adjacent atoms $I_f,J_f$ that belong to $\mathcal{Q}^f_n$. Estimate \eqref{coherencedfdm} is enough to prove that the derivative of the conjugacy $h$ is H\"older continuous on the whole circle (see \cite[Proposition 4.3\,(b), page 356]{dfdm1}), which is precisely Conclusion \eqref{Aitem3} of Theorem \ref{rigidity}.

\appendix

\section{Proof of Proposition \ref{holder}}\label{apA}

In this appendix we prove Proposition \ref{holder}. In the proof we follow the exposition of Lyubich in \cite[Lemma 11.5]{lyubich}:

\begin{proof}[Proof of Proposition \ref{holder}] Let $V$ be a bounded Jordan domain containing the interval $I$, and compactly contained in $U$ (as usual, a \emph{Jordan domain} is an open, connected and simply connected set of the complex plane, whose boundary is a Jordan curve). Consider a continuous function $h:\overline{V}\to[0,1]$ satisfying:
\begin{itemize}
\item $h$ is harmonic and positive in the annulus $V\!\setminus\! I$,
\item $h \equiv 0$ on $\partial V$ and $h \equiv 1$ on $I$.
\end{itemize}

Recall that the existence of such a function $h$ is a particular case of the \emph{Dirichlet's problem}.

To begin with the proof suppose first that $M=1$, and let $f:U\to\C$ be a holomorphic function such that $\big|f(z)\big| \leq 1$ for all $z \in U$. Let $\varepsilon=\|f\|_{C^0(I)} \leq 1$, and note that:
\begin{equation}\label{desborde}
\log|f| \leq h\log\varepsilon
\end{equation}
on $\partial(V\!\setminus\! I)=I\cup\partial V$. Since $f$ is holomorphic, $\log|f|$ is harmonic where $f \neq 0$ and subharmonic in the whole domain $V$, and since $h$ is harmonic in $V\!\setminus\! I$ we get from the maximum principle that inequality \eqref{desborde} also holds inside the annulus $V\!\setminus\! I$, that is, $\big|f(z)\big|\leq\varepsilon^{h(z)}$ for all $z \in V$.

Given $\alpha\in(0,1)$ let $W=\big\{z \in V:h(z)\in(\alpha,1]\big\}$, and note that $W$ is a Jordan domain containing $I$, compactly contained in $V$, and such that $h(z)=\alpha$ for all $z \in \partial W$. Since $\varepsilon\in[0,1]$ we have $\big|f(z)\big|\leq\varepsilon^{h(z)}\leq\varepsilon^{\alpha}$ for all $z \in W$, that is:$$\|f\|_{C^0(W)} \leq \big(\|f\|_{C^0(I)}\big)^\alpha.$$

The next step is just the standard application of Cauchy's integral formulas: let $\rho\in(0,1)$ such that $\overline{B(z,\rho)} \subset W$ for all $z \in I$. Then for any $z \in I$ and any $n\in\{0,1,...,r\}$ we have:
\begin{align}
\big|f^{(n)}(z)\big|&=\left|\frac{n!}{2\pi i}\int_{\partial B(z,\rho)}\frac{f(w)}{(w-z)^{n+1}}\,dw\right|=\frac{n!}{2\pi}\left|\int_{0}^{2\pi}\frac{f(z+\rho e^{i\theta})}{(\rho e^{i\theta})^{n+1}}i\rho e^{i\theta}d\theta\right|\notag\\
&\leq\frac{n!}{2\pi}\frac{1}{\rho^n}\int_{0}^{2\pi}\left|f(z+\rho e^{i\theta})\right|d\theta\leq\left(\frac{n!}{\rho^n}\right)\left(\sup_{w\in\partial B(z,\rho)}\left\{\big|f(w)\big|\right\}\right).\notag
\end{align}

Defining $L_1=r!/\rho^r$ we obtain:$$\|f\|_{C^r(I)} \leq L_1\|f\|_{C^0(W)} \leq L_1\big(\|f\|_{C^0(I)}\big)^\alpha.$$

Therefore Proposition \ref{holder} is true for the case $M=1$. For the general case note that for any $f\in\mathcal{F}$ we have $\|f\|_{C^r(I)}=M\|f/M\|_{C^r(I)} \leq ML_1\big(\|f/M\|_{C^0(I)}\big)^{\alpha}=M^{1-\alpha}L_1\big(\|f\|_{C^0(I)}\big)^{\alpha}$, and therefore is enough to consider $L=M^{1-\alpha}L_1$.
\end{proof}

\end{document}